\documentclass[12pt]{article}

\usepackage{amsthm,amsmath,amssymb,amsfonts,bbm}
\usepackage{tikz}
\usepackage{dsfont}
\usepackage{centernot}
\usepackage{graphicx}
\usepackage[abs]{overpic}
\usepackage{xcolor,varwidth}

\usepackage{ stmaryrd }

\newtheorem{lemma}{Lemma}[section]
\newtheorem{proposition}{Proposition}[section]
\newtheorem{corollary}{Corollary}[section]
\newtheorem{remark}{Remark}[section]
\newtheorem{claim}{Claim}[section]
\newtheorem{theorem}{Theorem}[section]

\newtheorem{definition}{\emph{Definition}}[section]

\newcommand{\pr}{\mathbb{P}}

\numberwithin{equation}{section}

\title{Metastability for the contact process with two types of particles and priorities}
\author{Mariela Pent\'on Machado}
\date{ }
\begin{document}
\maketitle
\begin{abstract}
We consider a symmetric finite-range contact process on $\mathbb{Z}$ with two types of particles (or infections), which propagate according to
the same supercritical rate and die (or heal) at rate 1. Particles of type $1$ can occupy any site in $(-\infty,0]$ that is empty or occupied by a particle of type $2$  and, analogously, particles of type $2$ can occupy any site in $[1,+\infty)$ that is empty or occupied by a particle of type $1$. We consider the model restricted to a finite interval $[-N+1,N]\cap \mathbb {Z}$. If the initial configuration  is $\mathds{1}_{(-N,0]}+2\mathds{1}_{[1,N)}$, we prove that this system exhibits two metastable states: one with the two species and the other 
one with the family that survives the competition.


\end{abstract}

\section{Introduction}
The aim of this work is the study of a metastable phenomenon for a stochastic process that can be interpreted as the time evolution of a population which has two different species and each of them has a favorable region in the environment.

A system is considered in a metastable state if it behaves as in a \emph{false equilibrium} distribution for a long random time until, abruptly, it gets to the true equilibrium.  Classical examples of this phenomenon include the behavior of  supercooled vapors and liquids, and supersaturated vapors and solutions. For a detailed discussion on metastability in stochastic processes and references, see the monographs \cite{Bovier} and \cite{olivieri2005large}.

A specific  stochastic process that fits into this situation is the  contact process, introduced by Harris in \cite{Harris}. It is a simple model for the spread of an infection, where individuals are identified with the vertices of a given graph which we may take as $\mathbb{Z}^d$. Every infected individual can propagate the infection to some neighbor at rate $\lambda$ and it becomes healthy at rate $1$.  This process presents a dynamical phase transition: there exists a critical value $\lambda_c$ for the infection rate such that if $\lambda$ is larger than $\lambda_c$, there is a non-trivial invariant measure $\mu$ different from $\delta_{\emptyset}$.  On the other hand,  when restricted to a finite volume, this  is a finite Markov chain and $\delta_{\emptyset}$ is the only equilibrium state. Nevertheless, for suitable initial conditions, the restriction of the non-trivial invariant measure to this finite volume behaves as a metastable state as described above. This was first proved in \cite{Casandro-Vares-Oliveri} for $\lambda$ sufficiently large and in the one-dimensional case, where the authors introduced a pathwise point of view for the study  of metastability in stochastic dynamics.  The basic idea of this approach  is to study the statistics of each path, performing time
averages along the evolution. This study includes basically two steps, first, it is proved that the time of extinction rescaled by its mean converges to an exponential distribution with mean $1$.  Secondly, they prove the convergence of suitably rescaled time averages along the evolution to a non-equilibrium distribution. This last  convergence is named thermalization property, and is clearly connected to the unpredictability of the transition out of the 'metastable state'. In \cite{Schonmman} this result was  extended to the whole supercritical region.  A different proof of the convergence of the time of extinction rescaled by its mean was proved in \cite{Schonmman-Durrett}, which also describes the asymptotic behavior of the logarithm of this time. These last results were extended for dimension $d\geq 2$ in \cite{Mountford} and \cite{Mountford2}, respectively. The thermalization property for the contact process in dimension $d\geq 2$ was proved in \cite{Simonis}.

 The contact process can be interpreted as the  time evolution of a certain population, where a site is now ``occupied" (in correspondence to ``infected") or ``empty"(in correspondence to``healthy"). We shall examine the one-dimensional situation but allow a propagation within distance $R>1$. There are some examples of processes inspired by the contact process that try to describe what happens if the population is not homogeneous, in the sense that some individuals have different characteristics. An  example is the process introduced in \cite{GBT} in which  every site in $\mathbb{Z}$ can be occupied by particles of type 1 or 2, but the particles of type 1 have priority throughout the environment.  We introduce a process in which the priority is no longer spatially homogeneous; particles of type 1 have priority in $(-\infty,0]\cap \mathbb{Z}$ and  particles of type $2$ in $[1,\infty)\cap \mathbb{Z}$. The process we are interested in is a continuous time Markov process with state space $\{0,1,2\}^{\mathbb{Z}}$ and we denote it by $\{\zeta_t\}_t$. If $\zeta_t(x)=i$, then the site $x$ is occupied at time $t$ by a particle of type $i$ (i=1,2) and if $\zeta_t(x)=0$ at time $t$, the site $x$ is empty. 
  We denote  the flip rates at  $x$ in a configuration $\zeta \in \{0,1,2\}^{\mathbb{Z}}$ by $c(x,\zeta, \cdot)$ and are defined as follows
$$ \begin{array}{ll}
 c(x, \zeta, 1\rightarrow 0 )= c(x, \zeta,2\rightarrow 0)= 1,\\
 c(x, \zeta, 0 \rightarrow i)=   \lambda \underset{y: \hspace{0.1cm} 0<|x-y|\leq R }{\sum} \mathds{1}_{\zeta(y)=i} , i= 1,2,\\
  c(x, \zeta, 2 \rightarrow 1)=   \lambda \underset{y: \hspace{0.1cm}  0<|x-y|\leq R }{\sum} \mathds{1}_{\zeta(y)=1}\mathds{1}_{\{x \in(-\infty,0]\}} ,\\
  c(x, \zeta, 1 \rightarrow 2)=   \lambda  \underset{y: \hspace{0.1cm}   0<|x-y|\leq R }{\sum} \mathds{1}_{\zeta(y)=2}\mathds{1}_{\{x \in [1,\infty)\}}.
\end{array}
$$
 We consider $R>1$ and restrict to the supercritical case,  where $\lambda>\lambda_{c}=\lambda_{c}(R)$. For most of the paper, we consider the initial configuration $\mathds{1}_{(-\infty,0]}+2\mathds{1}_{[1,\infty)}$.

In this paper, we prove that if the dynamic is restricted to an interval of length $N$, the time of the first extinction for one of the two populations, when properly rescaled, converges to the  exponential distribution as  $N$ tends to infinity.  We also prove a result that gives information on the asymptotic order of magnitude of this time (for the limit in $N$). Combining this result with the metastability of the classical contact process, we obtain that, after one of the species dies out, the surviving species lives during an exponential time. Since with only one type of particle the process behaves like the classical contact process, after one of the species dies out the process presents a new metastable state, which is the standard for the classical contact process.  

The paper is organized as follows. In Section \ref{General Setting}, we introduce the notation and state our main results. In Section \ref{Section3}, we define \emph{barriers} in a finite interval; this is a central tool in the development of the next sections. In Section \ref{Section4}, we present a result about the metastability for the classical contact process in dimension $1$ with range $R\geq1$. In Section \ref{Metastability}, we prove that the time of the first extinction in the interval converges  to an exponential distribution as the length of the interval  tends to infinity. In Section \ref{Section6}, we prove the convergence in probability of the logarithm of this time divided by the length of the interval to a positive constant.  

\section{Settings}\label{General Setting}
 In this section, we recall the Harris construction introduced in \cite{Harris}. Using this construction, we define the classical contact process. Also, using the Harris construction, we give another definition of the contact process with two types of particles and priorities restricted to the interval $[-N+1,N]$\footnote{We observe that in the introduction we gave a different definition of the contact process with two types of particles and priorities by defining  the rates of flips of the process.}. This definition provides a precise coupling between the classical contact process and the contact process with two types of particles and priorities (see Remark \ref{rem:1}).

In order to define the classical contact process with range $R\in \mathbb{N}$ and rate of infection $\lambda>0$, we consider a collection of independent Poisson point processes on $[0,\infty)$
\begin{align}\label{HarrisGraph}
\begin{split}
&\{P^{x}\}_{x \in \mathbb{Z}}\text{ with rate 1},\\&
\{P^{x\rightarrow y}\}_{\{x,y \in \mathbb{Z}: \hspace{0.2cm} 0<|x-y|\leq R\}} \text{ with rate }\lambda.
\end{split}
\end{align}

Graphically, we  place a cross mark at the point $(x,t)\in \mathbb{Z}\times [0, +\infty)$
whenever $t$ belongs to the Poisson process $P^x$. In addition, we place an arrow following the direction from  $x$ to $y$  whenever $t$ belongs to the Poisson process $P^{x\rightarrow y}$. We  denote by $\mathcal{H}$  the collection of these marks in $\mathbb{Z}\times[0,\infty)$, this is a \emph{Harris construction} (see Figure \ref{HarrisConstrution}). Given $(x,t)\in \mathbb{Z}\times[0,\infty)$, we denote by $\Theta_{(x,t)}(\mathcal{H})$ the Harris construction obtained by shifting $\mathcal{H}$ such that $(x,t)$ is the new origin.
 \begin{figure}[h]
	\centering
		  \includegraphics[scale=0.4,unit=1mm]{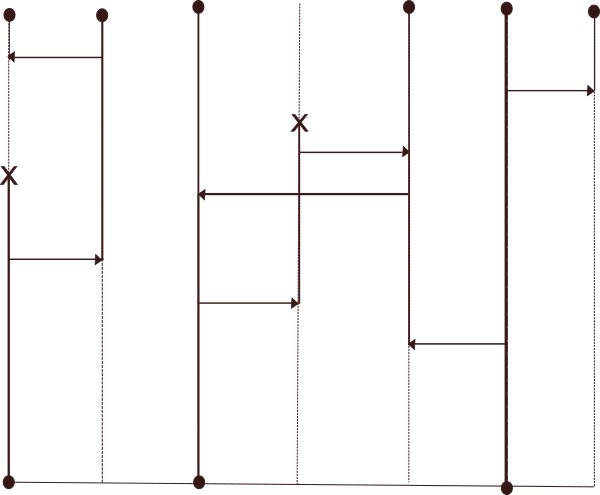}
   \caption{An example of a Harris construction for $R>1$.}\label{HarrisConstrution}
   \end{figure}

A \emph{path} in $\mathcal{H}$ is an oriented path which follows the  positive direction of time $t$, it passes along the arrows in the direction of them and does not pass through any cross mark. More precisely,  a path from $(x,s)$ to $(y,t)$, with $0<s<t$, is a piecewise constant function $\gamma:[s,t]\rightarrow \mathbb{Z}$ such that:
\begin{itemize}
 \item{$ \gamma(s)=x, \gamma(t)=y$},
 \item{$\gamma(r)\neq \gamma(r-)$ only if  $r \in P^{\gamma(r-)\rightarrow \gamma(r)}$}\footnote{The notation $r\in P^{x\rightarrow y}$ means that $r\in (0,\infty)$ is a jump time of the Poisson process $P^{x\rightarrow y}$. },
 \item{$\forall r \in [s,t], r \notin  P^{\gamma(r)}$.}
\end{itemize}
In this case, we say that $\gamma$ connects $(x,s)$ with $(y,t)$. Moreover, if such a path exists, we write $(x,s)\rightarrow(y,t)$.

 For  $A$, $B$  and  $C$ subsets of $\mathbb{Z}$ and $0\leq s<t$, we say that $A \times \{s \}$ is connected with $B\times\{t \} $ inside $C$, if there exist $x\in A$, $y \in B$ and a path $\gamma$ connecting $(x,s)$ with $(y,t)$ such that $\gamma( r)\in C$ for all $r$, $s\leq r\leq t$. We denote this situation by $A\times \{s\}\rightarrow B\times \{t\}$ inside $C$.

To simplify the notation,  throughout  the paper we identify  $I\cap\mathbb{Z}$ with $I$ for every spatial interval. Also, we identify every configuration $\xi$ in $\{0,1\}^{\mathbb{Z}}$ with the subset $\{x\in \mathbb{Z}: \xi(x)=1\}$.
 
 Given a Harris construction $\mathcal{H}$ and  a subset $A$   of  $\mathbb{Z}$, we define the classical contact process beginning at time $s$  with initial configuration $A$ as follows
\begin{equation}\label{Contactprocess}
\eta^{A}_{s,t}=\{x: \text{ exists } y \in A \text{ such that }(y,s)\rightarrow(x,s+t)\}.
\end{equation}
In the special case of $s=0$, we just write $\eta^{A}_t$. Furthermore,  we define the time of extinction of $ \eta^{A}_t$  as follows
\begin{equation}\label{timext}
T^{A}=\inf\{t>0: \eta^{A}_t=\emptyset\}.
\end{equation}
If  $A\subset [1,N]$,  the classical contact process restricted  to $[1,N]$ with initial configuration $A$  is denoted by 
\begin{equation}\label{ContactoSinColoresEnlaCaja}
\xi^{A,N}_{t}=\{x: \text{ exists } y \in A \text{ such that }(y,0)\rightarrow(x,t)\text{ inside }[1,N]\}.
\end{equation}
For this process, we define the time of extinction as follows
$$
T^{A}_{N}=\inf\{t>0: \xi^{A,N}_{t}=\emptyset\}.
$$

For the classical contact process in dimension $1$ with initial configuration $(-\infty,0]$, we denote the rightmost infected particle  by
\begin{equation}\label{rightmost}
r^{(-\infty,0]}_t=\max\{x: \,\exists \, y\in (-\infty,0]\text{ such that }(y,0)\rightarrow(x,t)\}.
\end{equation}
In \cite{Liggett} it is proved that for $R=1$ there exists $\alpha>0$ such that
\begin{equation}\label{rightmostedgespeed}
\frac{r^{(-\infty,0]}_t}{t} \underset{t\rightarrow\infty}{\longrightarrow}\alpha \text{ almost surely.}
\end{equation}
The above result is obtained using the Subadditive Ergodic Theorem and monotonicity arguments, and it can be adapted for the case $R>1$. This convergence result will be useful in the next section. 

Now we define the contact process with two types of particles and priorities restricted to $[-N+1,N]$ using the Harris construction $\mathcal{H}$. For $A$ and $B$ disjoint subsets of $[-N+1,N]$, we denote by $\{\zeta^{A,B,N}_t\}_t$  the contact process with two types of particles and priorities restricted to the interval $[-N+1,N]$, with initial configuration $\mathds{1}_{A}+2\mathds{1}_B$ and  the particles of type $1$ having priority in $[-N+1,0]$ and the type 2 having priority in $[1,N]$. In this case, it is simple to state the definition of this process in terms of  a Harris construction, since we are dealing with a stochastic process which has c\`adl\`ag trajectories with jumps only in the times of the Poisson processes $\{P^{x}\}_{x\in [-N+1,N ]}$ or $\{P^{y\rightarrow x}\}_{\{y,x \in [-N+1,N ]: \hspace{0.2cm} 0<|x-y|\leq R\}}$. Let  $t$ be one of those times, two scenarios are possible:\begin{itemize}
    \item[(1)]  $t\in P^{x} $ for some $x$. In this case, $x$ is empty at this time and we set $\zeta^{A,B,N}_t(x)=0$;
    \item[(2)] $t\in P^{y\rightarrow x}$ for some $x$ and $y$.  If $x$ is occupied by a particle of type $i$ ($i=1,2$), and $x$ is in the region of priority of this type of particles, then nothing changes at $x$. Otherwise, $x$ became occupied by the type of particle that is in $y$ and we set $\zeta^{A,B,N}_t(x)=\zeta^{A,B,N}_t(y)$.
\end{itemize}

We are interested in studying  the time in which one of the types of
particles die out and we denote that time by $\tau^{A,B}_N$. 
\begin{remark}\label{rem:1}
Since the classical contact process and the contact process with two types of particles and priorities are defined using the same Harris construction $\mathcal{H}$, both processes are defined in the same probability space. This coupling will be used in all the work. 
\end{remark}

For the sake of clarity, we now introduce several notations. During all the work we refer to the contact process with two types of particles and priorities as the two-type contact process. We denote by $\eta^{\textbf{1}}_t$ the classical contact process in $\mathbb{Z}$ with initial configuration $\mathbb{Z}$.  We denote by $\xi^{\textbf{1},N}_{t}$   the classical contact process restricted to $[1,N]$ with  initial configuration  $[1,N]$ and denote the extinction time of this process by $T^{\textbf{1}}_N$. Furthermore,  we denote by $\xi^{x,N}_{t}$ the classical contact process with initial configuration $\{x\}$ and its extinction time by  $T^{x}_{N}$. For the initial configuration $\mathds{1}_{[-N+1,0]}+2\mathds{1}_{[1,N]}$, we denote the two-type contact process restricted to $[-N+1,N]$ by  $\zeta^{\textbf{1},\textbf{2},N}_{t}$ and $\tau^{\textbf{1},\textbf{2}}_{N}$ is the time when one of the  types of particles dies out.
We stress that, during the paper, the letters $\xi$ and $\eta$ refer to the classical contact process and $\zeta$ refers to the two-type contact process.

 Now, we are ready to enunciate the main results of this work.

\begin{theorem}\label{TeoMetaD=1R>1}
Let $\beta_N$ be such that $\pr(\tau^{\textbf{1},\textbf{2}}_N\geq \beta_N)=e^{-1}$, then
$$
\underset{N\rightarrow \infty}{\lim}\frac{\tau^{\textbf{1},\textbf{2}}_N}{\beta_N}=E \text{ in Distribution},
$$
where $E$ has exponential distribution with rate $1$.
\end{theorem}

\begin{theorem}\label{Theorem1}
There exists a constant $c_{\infty}>0$ depending only on the rate of infection $\lambda$ and the range $R$ such that
$$\underset{N\rightarrow \infty}{\lim}\frac{1}{N}\log\tau^{\textbf{1},\textbf{2}}_N=c_{\infty}\text{ in Probability}.$$
\end{theorem}

\section{Barriers in finite volume}\label{Section3}
In  this section, we introduce the definition of  an $N$\emph{-barrier} which is similar to the notion called descendancy barrier introduced in  \cite{Conos}.  The main difference between these two concepts is that the $N$-\emph{barrier} is defined for the classical contact process in an interval whose length depends on $N$, while the descendancy barrier is defined in the whole line. This section is devoted to establishing some properties of $N$\emph{-barriers} and follows closely Section $2.2$ of \cite{Conos}.

 The main idea behind the structure we introduce here is to extend for the classical contact process with range $R>1$ the following property that holds in the case $R=1$: Fix $x$ in $[1,N]$, $D>0$, $t\in[DN^2,2DN^2]$ and consider the event
$$\{(x,0)\rightarrow(1,t) \text{ inside }[1,N] ; (x,0)\rightarrow(N,t) \text{ inside }[1,N]\}.$$ 
By the path crossing property, we have in this event that $\xi^{x}_N(t)=(\eta^{\mathbb{Z}}_t\cap[1,N])$. Furthermore, Corollary $1$ in \cite{MountfordSweet} establishes that for any $D>0$ there is a constant $\delta>0$ such that
$$
\pr((z,0)\rightarrow(y,t) \text{ inside }[1,N])\geq \delta,
$$
for any $z,y\in[1,N]$, $t\in[DN^2, 2DN^2]$ and sufficiently large $N$. Now, by the $FKG$-inequality  we have 
$$
\pr((x,0)\rightarrow(1,t) \text{ inside }[1,N] ; (x,0)\rightarrow(N,t) \text{ inside }[1,N])\geq \delta^2,
$$
which implies that there exists $\hat{\eta}=\delta^{2}$ such that
\begin{equation}\label{N-barrierD=1}
\underset{x\in[1,N]}{\inf}\pr(\xi^{x}_N(t)=(\eta^{\textbf{1}}_t\cap[1,N]))>\hat{\eta}>0.
\end{equation}

The strong use of the path crossing property to obtain \eqref{N-barrierD=1} restricts this argument to the case $R=1$. In Proposition \ref{N-barrier} we extend \eqref{N-barrierD=1} to the classical contact process with range $R>1$. For this purpose,  we introduce the definition of  an $N$\emph{-barrier}. The main tool behind the construction of  an $N$-\emph{barrier} is the Mountford-Sweet renormalization  introduced in  \cite{MountfordSweet}, which we briefly  discuss now. To this end, we first recall some notions of oriented percolation.

Consider $\Lambda=\{(m,n)\in \mathbb{Z}\times \mathbb{Z}^{+}: m+n \text{ is even }\}$, $\Omega=\{0,1\}^{\Lambda}$ and $\mathcal{F}$ the $\sigma$-algebra generated by the cylinder sets of $\Omega$. Also, we consider $\mathcal{F}_n$ the $\sigma$-algebra generated by the cylinder sets of $\Omega$ that depend on points $(m,s)\in \Lambda$ with $s\leq n$.
 
 Given $\Psi \in \Omega$, we say that two points $(m,k)$, $(m',k')\in\Lambda$ with $k<k'$ are \emph{connected by an open path} (\emph{according to} $\Psi$) \cite{Conos}, if there exists a sequence $\{(m_i, n_i)\}_{0\leq i\leq k'-k}$  such that 
 $$(m_0,n_0)=(m, k),\quad (m_{k'-k},n_{k'-k})=(m', k'),\quad |m_{i+1}-m_{i}|=1,\quad  n_i=k+i,$$ with $0\leq i\leq k'-k-1$ and $\Psi(m_i, n_i)=1$ for all $i$. 
 If $(m,k)$ and $(m',k')$ are connected by an open path (according to $\Psi$), we write $(m,k)\rightsquigarrow (m',k')$ (according to $\Psi$).

Now, let $A$, $B$ and $C$ be subsets of $\Lambda$. We say that $A \times \{n \}$ is connected with $B\times\{n' \} $  inside $C$, if there are $m\in A$ and $m' \in B$ such that $(m,n)\rightsquigarrow (m',n')$ and all the edges of the path are in $C$. In this case, we write $A\times\{n\}\rightsquigarrow B\times\{n'\} \text{ inside }C$.

Given $k\geq 1$ and $\delta>0$,  $(\Omega,\mathcal{F}, \hat{\mathbb{P}})$ is a $k$-dependent oriented percolation system with closure below $\delta$, if  for all $r$ positive
\begin{small}
$$
\hat{\mathbb{P}}( \Psi(m_i,n)=0, \forall i\hspace{ 0.2 cm} 0\leq i\leq r|\mathcal{F}_{n-1})<\delta^r,
$$
\end{small}with $(m_i,n)\in \Lambda$ and $|m_i-m_j|>2k$ for all $i \neq j$ and $1 \leq i,j \leq r$ (see \cite{MountfordSweet}, \cite{Conos}).

Let $\Psi$ and $\Psi'$ be two elements of $\Omega$, we say that  $\Psi\leq \Psi'$  if $\Psi(m,n)\leq\Psi'(m,n)$, for all $(m,n)\in \Lambda$. Also,  we say that a subset  $A$ of $\{0,1\}^\Lambda$ is increasing if $\Psi \in A$ and $\Psi\leq\Psi'$,  then $\Psi' \in A$.  Let $\hat{\mathbb{P}}_1$ and $\hat{\mathbb{P}}_2$ be two measures on $\mathcal{F}$, we say that $\hat{\mathbb{P}}_1$ stochastically dominates $\hat{\mathbb{P}}_2$ if $\hat{\mathbb{P}}_1(A)\geq\hat{\mathbb{P}}_2(A)$ for all $A$ increasing in $\mathcal{F}$.

  The next result follows via the dual-contours methods of Durrett; for details see \cite{Survey}.  
\begin{lemma}\label{surveyD} Let $\hat{\mathbb{P}}_{p}=\prod_{\Lambda} (p\delta_1+(1-p)\delta_0)$ be the Bernoulli product measure on $\Lambda$. There exist $\epsilon_{1}$ and $p_0$ such that for all $p>p_0$:
	  $$
	        \underset{\begin{tiny}\begin{array}{c}x\in[1,M], \hspace{0.1cm}x \text{ even};\\
	         y \in [1,M],\hspace{0.1cm} y+M^{2} \text{ even}  
	    \end{array}\end{tiny}}{\inf} \hat{\mathbb{P}}_{p}((x,0)\rightsquigarrow(y,M^{2}) \text{ inside } [1,M] )\geq \epsilon_1,
	    $$
for every $M$ large enough. 
	\end{lemma}
The following lemma is a consequence of Theorem $0.0$ in \cite{ LiggetSchonmannStacey}  and allows us to extend Lemma \ref{surveyD} to $k$\emph{-dependent} percolation systems with small closure.
\begin{lemma}\label{LiggettSmallClosure}
Consider $\hat{\mathbb{P}}_{p}$ the Bernoulli product measure on $\Lambda$. For  $k\in \mathbb{N}$  and  $0<p<1$ fixed, there exists $\delta>0$ such that if $(\Omega,\mathcal{F},\hat{\pr})$ is a  k-dependent oriented percolation system with closure below   $\delta$,  then    $\hat{\mathbb{P}}$  stochastically dominates $\hat{\mathbb{P}}_{p}$.
\end{lemma}

 	We now  define  a measurable map $\Psi$, with state space  $\Omega$, introduced in \cite{MountfordSweet}. The definition of this map depends on two positive integers $\hat{N}$ and $\hat{K}$. In \cite{MountfordSweet} it is proved that for any $\delta$ it is possible to choose $\hat{N}$ and $\hat{K}$ such that the law of $\Psi$ is a   $k$\emph{-dependent} percolation system  with closure under $\delta$. 
 	
 	Let $\hat{N}$ and $\hat{K}$ be two positive integers. 
		Given $m\in \mathbb{Z}$ and $n \in \mathbb{Z}^{+}$ such that $m+n$ is even, we define the following sets 
\begin{align*}
&\mathcal{I}^{\hat{N}}_m=\left(\frac{m\hat{N}}{2}-\frac{\hat{N}}{2},\frac{m\hat{N}}{2}+\frac{\hat{N}}{2}\right]\cap \mathbb{Z},\\&
I^{\hat{N},\hat{K}}_{(m,n)}=\mathcal{I}^{\hat{N}}_m\times \{\hat{K}\hat{N}n\},\\&
J^{\hat{N},\hat{K}}_{(m,n)}=\left(\frac{m\hat{N}}{2}-R, \frac{m\hat{N}}{2}+R\right)\times [\hat{K}\hat{N}n, \hat{K}\hat{N}(n+1)].
\end{align*}
We call the set 
$$
I^{\hat{N},\hat{K}}_{(m,n)}\cup J^{\hat{N},\hat{K}}_{(m,n)}\cup I^{\hat{N},\hat{K}}_{(m,n+1)}
$$
the renormalized box corresponding to $(m,n)$, or just the box $(m,n)$.

We start defining an auxiliary  $\Phi\in \{0,1,2\}^{\Lambda}$. Given $(m,0)\in \Lambda$, put $\Phi(m,0)=1$ if the following conditions are satisfies 
\begin{equation}\label{MSa0}
\begin{array}{c}
\text{ For each interval }I\subset \mathcal{I}^{\hat{N}}_{m-1} \cup \mathcal{I}^{\hat{N}}_{m+1} \text{ of length } \sqrt{\hat{N}},\\ \text{ it holds } I \cap \eta^{\textbf{1}}_{\hat{N}\hat{K}}\neq \emptyset;
\end{array}
\end{equation}

\begin{equation}\label{MSb0}
\begin{array}{c}
\text{ If }x\in  \mathcal{I}^{\hat{N}}_{m-1} \cup \mathcal{I}^{\hat{N}}_{m+1}\text{ and }\mathbb{Z}\times{\{0\}}\rightarrow (x,\hat{K}\hat{N}),\\
\text{ then } I^{\hat{N},\hat{K}}_{(m,0)}\rightarrow (x,\hat{K}\hat{N});
\end{array}
\end{equation}

\begin{equation}\label{MSc0}
\begin{array}{c}
\text{If }(x,s) \in J^{\hat{N},\hat{K}}_{(m,0)} \text{ and }\mathbb{Z}\times{\{0\}}\rightarrow (x,s),\\
\text{ then }I^{\hat{N},\hat{K}}_{(m,0)}\rightarrow (x,s);
\end{array}
\end{equation}

\begin{small}
\begin{align}\label{MSd0}
\begin{split}
    &\left\lbrace\begin{array}{c}
x\in \mathbb{Z}:\exists s, t,0\leq s<t\leq \hat{K}\hat{N},\\
 y\in  \mathcal{I}^{\hat{N}}_{m-1} \cup \mathcal{I}^{\hat{N}}_{m+1}  \text{ such that } (x,s)\rightarrow (y,t)
\end{array}\right\rbrace  \\& \hspace{5cm}\subset\left[ \frac{m\hat{N}}{2}-2\alpha \hat{K}\hat{N}, \frac{m\hat{N}}{2}+2\alpha \hat{K}\hat{N} \right].
\end{split}
\end{align}\end{small}put $\Phi(m,0)=0$ otherwise. Given $(m,n)\in \Lambda$ with $n\geq 1$, put $\Phi(m,n)=1$ if
\begin{equation}\label{Msa'}
1\in \{\Phi(m-1,n-1),\Phi(m+1,n-1)\};
\end{equation}

\begin{equation}\label{MSa}
\begin{array}{c}
\text{ For each interval }I\subset \mathcal{I}^{\hat{N}}_{m-1} \cup \mathcal{I}^{\hat{N}}_{m+1} \text{ of length } \sqrt{\hat{N}},\\ \text{ it holds } I \cap \eta^{\textbf{1}}_{\hat{N}(n+1)}\neq \emptyset;
\end{array}
\end{equation}

\begin{equation}\label{MSb}
\begin{array}{c}
\text{ If }x\in  \mathcal{I}^{\hat{N}}_{m-1} \cup \mathcal{I}^{\hat{N}}_{m+1}\text{ and }\eta^{\textbf{1}}_{\hat{K}\hat{N}n}\times{\{\hat{K}\hat{N}n\}}\rightarrow (x,\hat{K}\hat{N}(n+1)),\\
\text{ then }(\eta^{\textbf{1}}_{\hat{K}\hat{N}n}\times{\{\hat{K}\hat{N}n\}})\cap I^{\hat{N},\hat{K}}_{(m,n)}\rightarrow (x,\hat{K}\hat{N}(n+1));
\end{array}
\end{equation}

\begin{equation}\label{MSc}
\begin{array}{c}
\text{If }(x,s) \in J^{\hat{N},\hat{K}}_{(m,n)} \text{ and }\eta^{\textbf{1}}_{\hat{K}\hat{N}n}\times{\{\hat{K}\hat{N}n\}}\rightarrow (x,s),\\
\text{ then }(\eta^{\textbf{1}}_{\hat{K}\hat{N}n}\times{\{\hat{K}\hat{N}n\}})\cap I^{\hat{N},\hat{K}}_{(m,n)}\rightarrow (x,s);
\end{array}
\end{equation}

\begin{small}
\begin{align}\label{MSd}
\begin{split}
    &\left\lbrace\begin{array}{c}
x\in \mathbb{Z}:\exists s, t,\hat{K}\hat{N}n\leq s<t\leq \hat{K}\hat{N}(n+1),\\
 y\in  \mathcal{I}^{\hat{N}}_{m-1} \cup \mathcal{I}^{\hat{N}}_{m+1}  \text{ such that } (x,s)\rightarrow (y,t)
\end{array}\right\rbrace  \\& \hspace{5cm}\subset\left[ \frac{m\hat{N}}{2}-2\alpha \hat{K}\hat{N}, \frac{m\hat{N}}{2}+2\alpha \hat{K}\hat{N} \right].
\end{split}
\end{align}\end{small}
If \eqref{Msa'} fails put $\Phi(m,n)=2$, and in every other case put $\Phi(m,n)=0$. Finally, set 

\begin{equation}
\Psi(m,n)=\left\lbrace
\begin{array}{cl}
0, & \text{ if }\Phi(m,n)=0\\
1, & \text{ otherwise}.
\end{array}\right.
\end{equation}

We now make several remarks about the conditions in the definition of $\Phi$. First, equation \eqref{MSa} implies that there are many sites on the base of the boxes  $(m-1,n)$ and $(m+1,n)$ which are connected in the Harris construction with $\mathbb{Z}\times\{ 0\}$. Second, equation \eqref{MSb} yields that if  a site at the top of the box  $(m,n)$ is connected in the Harris construction with $\mathbb{Z}\times \{0\}$, then it is connected with the base of the box $(m,n)$. Third, equation  \eqref{MSc} guarantees that if a  site in the rectangle $J^{\hat{N},\hat{K}}_{(m,n)}$ is connected with $\mathbb{Z}\times \{0\}$, then it is connected with the base of the  box $(m,n)$. Finally, equation \eqref{MSd} implies that every path connecting a site in the  box  $(m,n)$ with $\mathbb{Z}\times\{ 0 \}$  is inside the rectangle 
\begin{equation}\label{rect}\left[ \frac{m\hat{N}}{2}-2\alpha \hat{K}\hat{N}, \frac{m\hat{N}}{2}+2\alpha \hat{K}\hat{N} \right]\times[\hat{K}\hat{N}n,\hat{K}\hat{N}(n+1)].
\end{equation}
 The  rectangle in \eqref{rect} is called the envelope of the box $(m,n)$.

Additionally, we observe  that the constant  $\alpha$ in equation \eqref{MSd} is as in \eqref{rightmostedgespeed}.

 The following proposition shows that we can construct  $\Psi$ with sufficiently small closure. Its proof can be found in \cite{MountfordSweet}.
\begin{proposition}\label{Mountford-Sweet}
There exist $k$ and $\hat{K}$ with the property that, for any $\delta>0$ there is $\hat{N}_0$ such that the law of $\Psi$ is a $k$-dependent percolation system with closure under $\delta$ for all $\hat{N}>\hat{N}_0$.
\end{proposition}Throughout the paper we fix 
\begin{itemize}
    \item[*]$k$ and $\hat{K}$ as in Proposition \ref{Mountford-Sweet};
    \item[*]$p_0$ as in Lemma $\ref{surveyD}$;
    \item[*]$\delta=\delta(k,p_0)$ as in Lemma \ref{LiggettSmallClosure};
    \item[*]$\hat{N}_0=\hat{N}_0(\delta,k,\hat{K})$ as in Proposition \ref{Mountford-Sweet};
    \item[*]$\hat{N}>\hat{N}_0$.
\end{itemize}
We note that these conditions imply that the law of  $\Psi$ is a $k$\emph{-dependent} percolation system with closure under $\delta$ and it is stochastically larger than $\hat{\pr}_{p_0}$.

Now, we are ready to introduce the definition  of an $N$-\emph{barrier}.
\begin{definition}\label{DefinitionN-barrier}
 Consider $M=M(N)=\left\lfloor \frac{2(N-2\alpha \hat{K}\hat{N})}{\hat{N}}\right\rfloor$ and $S=S(N)=\hat{K}\hat{N}M^{2}+2$. 
\begin{itemize}
\item[(a)] For $x\in[1,N]$ we say that $(x,0)$ is an $N$-barrier if  for all $y \in[1,N]$ such that $[1,N]\times\{0\}\rightarrow (y,S)$, we have that $(x,0)\rightarrow(y,S)$ inside $[1,N]$.

\item[(b)] For $x\in[-N+1,0]$ we say that $(x,0)$ is an $N$-barrier if  for all $y \in[-N+1,0]$ such that $[-N+1,0]\times\{0\}\rightarrow (y,S)$, we have that $(x,0)\rightarrow(y,S)$ inside $[-N+1,0]$. 
\item[(c)] For $x \in [-N+1,N]$ we say that a point $(x,t)$ is an $N$-\emph{barrier} if $(x,0)$ is an $N$-barrier in $\Theta_{(0,t)}(\mathcal{H})$.
\end{itemize}
\end{definition}
We note that, for large $N$, $M$ in Definition \ref{DefinitionN-barrier} is the largest $m$ such that the envelope of the box $(m,0)$ is a subset of $[1,N]\times[0,\infty)$.  

Our next step is to prove that, for $N$ large enough,  the probability of a point $(x,0)$ to be an $N$-\emph{barrier} is uniformly  bounded away from zero (Proposition \ref{N-barrier} below). To this end, we first introduce some notations.

Define the following sets
\begin{align}\label{Prop 0 eq0}
\begin{split}
&A_1=\left[1,\frac{\imath\hat{N}}{2}-\frac{\hat{N}}{2}\right),\quad A_2=\left[\frac{\imath\hat{N}}{2}-\frac{\hat{N}}{2},\frac{M\hat{N}}{2}+\frac{\hat{N}}{2}\right],\\
&A_3=\left(\frac{M\hat{N}}{2}+\frac{\hat{N}}{2},N\right],
\end{split}
\end{align}
where
\begin{equation}\label{imath(m)}
		    \imath=\imath(N)=\left\lbrace \begin{array}{ll}
		    \lfloor 4\alpha \hat{K}\hat{N} \rfloor & \text{ if } M^2+\lfloor 4\alpha \hat{K}\hat{N} \rfloor \text{ is even},\\
		    \lfloor 4\alpha \hat{K}\hat{N} \rfloor+1 & \text{ if }M^2+\lfloor 4\alpha \hat{K}\hat{N}\rfloor\text{ is odd},\end{array}\right.
		\end{equation}
and observe that the collection $\{A_i\}_{\{i=1,2,3\}}$ is a partition of the interval $[1,N]$.

 Now, for $x\in A_2$ and $j$ such that  $x\in\mathcal{I}^{\hat{N}}_{j}$ define
\begin{itemize}
%
\item[$E_1$=] $\{ \forall z \in  \mathcal{I}^{\hat{N}}_{j-1}\cup\mathcal{I}^{\hat{N}}_{j}\cup\mathcal{I}^{\hat{N}}_{j+1},\hspace{0.2cm}(x,0)\rightarrow(z,1)\text{ inside }\mathcal{I}^{\hat{N}}_{j-1}\cup\mathcal{I}^{\hat{N}}_{j}\cup\mathcal{I}^{\hat{N}}_{j+1}\times[0,1]\}$; 
\item[$E_2$=] $\{\Psi(\Theta_{(0,1)}(\mathcal{H}))\in \Gamma_M(j)\}$;
\item[$E_3$=] \begin{small}
\begin{equation*}\bigcap_{x,y \in A_1 \cup A_3} \{P^{x\rightarrow y}\cap (S-1,S]=\emptyset; \hspace{0.2cm}P^{y\rightarrow x}\cap (S-1,S]=\emptyset;
\hspace{0.2cm} P^{x}\cap(S-1,S]\neq\emptyset\};
\end{equation*} 
\end{small}
\end{itemize}
where $\Gamma_M$ is defined as follows
$$\begin{small}\Gamma_M(j)=\left\lbrace(j,0)\rightsquigarrow (\imath,M^2)\text{ and }(j,0)\rightsquigarrow(M,M^2) \text{ inside }\Lambda\cap([\imath,M]\times[0,M^2])  \right\rbrace\end{small}.$$

 We use Figure \ref{NBarrier} below to describe the event $\cap^{3}_{i=1}E_i$. Event $E_1$ guarantees that  every point in $\mathcal{I}^{\hat{N}}_{j-1}\cup\mathcal{I}^{\hat{N}}_{j}\cup\mathcal{I}^{\hat{N}}_{j+1}$ is connected with $(x,0)$ inside $\mathcal{I}^{\hat{N}}_{j-1}\cup\mathcal{I}^{\hat{N}}_{j}\cup\mathcal{I}^{\hat{N}}_{j+1}\times[0,1]$. In the right corner of Figure \ref{NBarrier} we represent an example of how event $E_1$ can occur. Event $E_2$ implies that there are two renormalized paths connecting the box $(j,0)$ with the boxes $(\imath,M^2)$ and $(M,M^2)$. These renormalized paths are represented in the figure by the red connected structure  $B_M$ (we  define $B_M$ formally  in equation \eqref{B_M} below). Finally, event $E_3$ ensures  that there are no particles in the regions $A_1\times[S-1,S]$ and $A_3\times[S-1,S]$ and this event is represented in the figure with two gray rectangles at the top. 

In the next proposition we prove that  for all configuration in $\cap^{3}_{i=1}E_i$, $(x,0)$ is an $N$-\emph{barrier}. 

\begin{figure}[ht!]
\begin{center}
  \begin{overpic}[scale=1.0,unit=1mm]{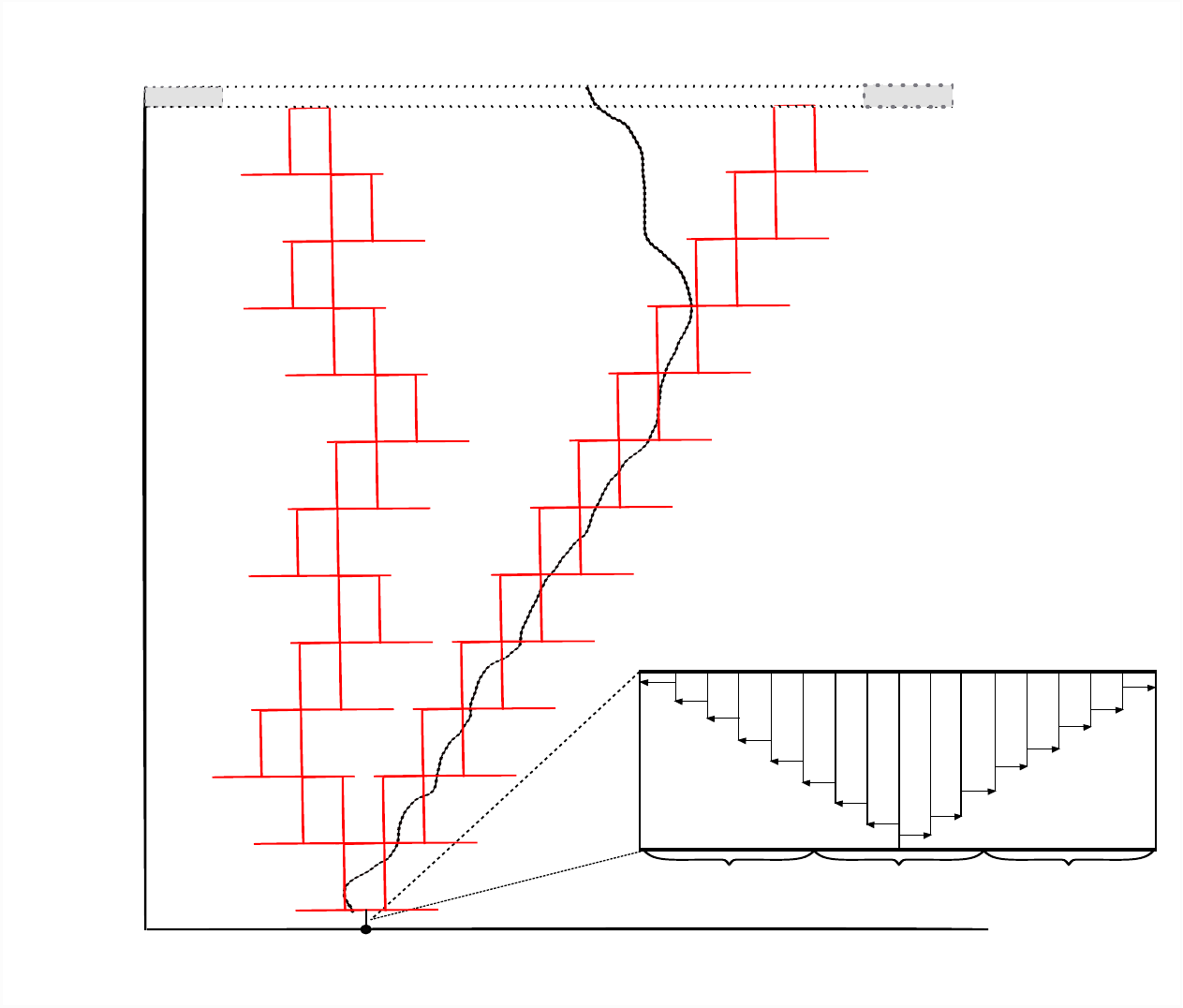}
   \put(18,101){\parbox{0.4\linewidth}{%
\begin{tiny}
$A_1$
\end{tiny}}}
\put(55,101){\parbox{0.4\linewidth}{%
\begin{tiny}
$A_2$
\end{tiny}}}
\put(95,101){\parbox{0.4\linewidth}{%
\begin{tiny}
$A_3$
\end{tiny}}}
\put(13,100){\parbox{0.4\linewidth}{%
\begin{tiny}
$S$
\end{tiny}}}
\put(22,92){\parbox{0.4\linewidth}{%
\begin{tiny}
$(\imath,M^{2})$
\end{tiny}}}
\put(45,42){\parbox{0.4\linewidth}{%
\color{red}{$B_M$}}}

\put(89.5,92){\parbox{0.4\linewidth}{%
\begin{tiny}
$(M,M^{2})$
\end{tiny}}}
 \put(31,11){\parbox{0.4\linewidth}{%
\begin{tiny}
$(j,0)$
\end{tiny}}}   
\put(36,6){\parbox{0.4\linewidth}{%
\begin{tiny}
$(x,0)$
\end{tiny}}}
\put(76, 12){\parbox{0.4\linewidth}{%
\begin{tiny}
$I^{\hat{N},\hat{K}}_{(-2,0)}$
\end{tiny}}}
\put(94.5, 12){\parbox{0.4\linewidth}{%
\begin{tiny}
$I^{\hat{N},\hat{K}}_{(0,0)}$
\end{tiny}}}
\put(112.5, 12){\parbox{0.4\linewidth}{%
\begin{tiny}
$I^{\hat{N},\hat{K}}_{(2,0)}$
\end{tiny}}}
\end{overpic}
\end{center}
\caption{Respresentation of the event $\cap^{3}_{i=1}E_i$.}\label{NBarrier}
\end{figure}
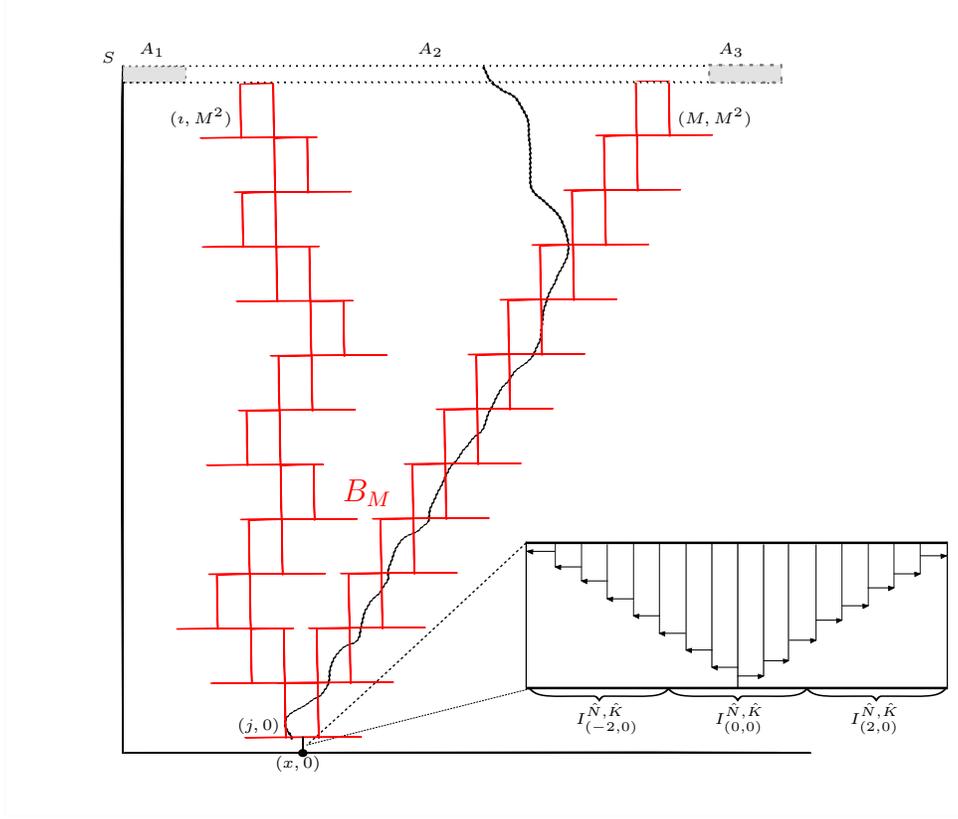

\begin{proposition}\label{N-barrier}
		 There	exists $\hat{\eta}=\hat{\eta}(\lambda)>0$ such that for all $N$ large enough 
\begin{equation}\label{claim 1 eq1}
    \pr^{x}((x,0)\text{ is an } N\text{-barrier})>\hat{\eta},
\end{equation}	
for any   $x\in\left[\frac{-M\hat{N}}{2}-\frac{\hat{N}}{2},\frac{-\imath\hat{N}}{2}+\frac{\hat{N}}{2}\right]\cup \left[\frac{\imath\hat{N}}{2}-\frac{\hat{N}}{2},\frac{M\hat{N}}{2}+\frac{\hat{N}}{2}\right]$.
\end{proposition}
	\begin{proof}
The case $R=1$ was discussed at the beginning of this section. We only need to  observe that for $N$ large enough $\frac{3}{2\hat{N}}N\leq M\leq \frac{2}{\hat{N}}N$ and $\hat{K}\frac{9}{4\hat{N}}N^2\leq S\leq\hat{K}\frac{4}{\hat{N}}N^2$. Therefore, for $D=\hat{K}\frac{9}{4\hat{N}}$, $DN^{2}\leq S\leq 2DN^{2}$ 
 and equation  \eqref{N-barrierD=1} implies  \eqref{claim 1 eq1}.

 The case $R>1$ is more complicated because the process does not have the path crossing property. Let us prove this case.
 
 For a configuration in $E_2$ there exist  sequences $\{m_k\}_{0\leq k\leq M^2}$ and $\{\hat{m}_k\}_{0\leq k\leq M^2}$, subsets of $\{\imath,\dots,M\}$, such that
$$m_0=\hat{m}_0=j,\quad m_{M^2}=\imath, \quad\hat{m}_{M^2}=M,$$ 
and
$$
\begin{array}{c}
\Psi(\Theta_{(0,1)}(\mathcal{H}))(m_k,k)=1\quad \text{   and   }\quad |m_{k+1}-m_{k}|=1 \hspace{0.2cm}\forall\hspace{0.2cm} k \hspace{0.2cm}\in \{0,\dots,M^{2}\},\\\Psi(\Theta_{(0,1)}(\mathcal{H}))(\hat{m}_k,k)=1\quad \text{   and   }\quad|\hat{m}_{k+1}-\hat{m}_{k}|=1\hspace{0.2cm} \forall \hspace{0.2cm} k\hspace{0.2cm}\in \{0,\dots,M^{2}\}.
\end{array}
$$
Denote 
\begin{equation}\label{B_M}
 B_{M}=\underset{0\leq k \leq M^{2}}{\bigcup}I^{\hat{K},\hat{N}}_{(m_k,k)}\cup J^{\hat{K},\hat{N}}_{(m_k,k)}\cup I^{\hat{K},\hat{N}}_{(\hat{m}_k,k)}\cup J^{\hat{K},\hat{N}}_{(\hat{m}_k,k)}.
 \end{equation}

From properties \eqref{MSb} and \eqref{MSc} in the definition of $\Psi$, it  follows that, in the trajectory of the classical contact process $t\mapsto \eta_t(\Theta_{(0,1)}(\mathcal{H}))$, every occupied site in $B_{M}$  descends from $I^{\hat{K},\hat{N}}_{(j,0)}$.  By our choice of $N$ and  $\imath$, we have that 
$$\left[ \frac{m_k \hat{N}}{2}-2\alpha \hat{K}\hat{N}, \frac{m_k\hat{N}}{2}+2\alpha \hat{K}\hat{N} \right]\subset[1,N]\text{ for all }k.$$ 
Therefore, the envelopes of the renormalized sites $(m_k,k)$ and $(\hat{m}_k,k)$ are subsets of $[1,N]\times[0,\infty)$ for all $k$. Using property \eqref{MSd} of the Mountford-Sweet renormalization, we have that every occupied site in $B_{M}$ is connected with $I^{\hat{K},\hat{N}}_{(j,0)}$  by a path entirely contained in $[1,N]\times [0,\infty)$.

We observe that $B_M$ is a connected union of $M^{2}$ segments of length $\hat{N}$ with rectangles of width $2R$ and height $\hat{K}\hat{N}$. Therefore, in $\Theta_{(0,1)}(\mathcal{H})$, at time $\hat{K}\hat{N}M^{2}=S-2$ every occupied site in $A_2$ is connected with $\mathbb{Z}\times\{0\}$  by a path that intersects the structure $B_M$ and remains in $[1,N]\times[0,\infty)$ afterward. Since every point in $B_M$ is connected with $I^{\hat{K},\hat{N}}_{(j,0)}$ inside $[1,N]$, we also can connect every point in $A_2\times\{S-2\}$ with $I^{\hat{K},\hat{N}}_{(j,0)}$ inside $[1,N]\times[0,\infty)$, in the construction $\Theta_{(0,1)}(\mathcal{H})$.
 
 The  event $E_1$ implies that every point in $\mathcal{I}^{\hat{N}}_{j}\times\{1\}$ is connected with $(x,0)$ in $\mathcal{H}$. Since $\mathcal{I}^{\hat{N}}_{j}\times\{1\}$ is the base of $I^{\hat{K},\hat{N}}_{(j,0)}$, we  have that  for all  $y \in A_2$ such that $\mathbb{Z}\times\{0\}\rightarrow(y,S-1)$,  $(x,0)\rightarrow(y,S-1)$ inside $[1,N]$.

Finally, for any realization in  $E_3$  there  is no mark of infection  in  the regions $A_1\times[S-1,S]$ and $A_3\times[S-1,S]$, and also there is no mark of infection going out or coming in these regions.  In particular, for any initial configuration at time $S$ there is no particle alive in $A_1\cup A_3$, and during the interval of time $[S-1,S]$ there is no interaction with any exterior region. Therefore,  every occupied site at time $S$ is connected with $A_2 \times\{S-1\}$ inside $A_2$ and we can conclude that every occupied site  in $[1,N]$ at time $S$ is connected with $(x,0)$ inside $[1,N]$, which is the definition of $N$-barrier.

Now we proceed to prove that the probability of $\cap^3_{i=1} E_i$ is positive. It is trivial that we can take $\tilde p>0$ independent of $N$ such that $P(E_1)\geq \tilde p$.	
Since the event $\{\Psi(\Theta_{(0,1)}(\mathcal{H}))\in \Gamma_M(j)\}$ depends on the Harris construction restricted to $\mathbb{Z}\times[1,S-1)$, we have that it is independent of all the marks in $\mathbb{Z}\times[0,1)$.
	Let us prove  that the event  $\{\Psi(\Theta_{(0,1)}(\mathcal{H}))\in \Gamma_M(j)\}$ has a positive probability.
	
 Using the $FKG$-inequality and Lemma \ref{surveyD} we have that for all $M$ large enough 
$$
\hat{\pr}_{p_0}((j,0)\rightsquigarrow (\imath,M^2)\text{ and }(j,0)\rightsquigarrow(M-1,M^2) \text{ inside }[\imath,M])>{\epsilon_1}^{2}.
$$
By our choice of $\Psi$, the law of $\Psi$ is stochastically larger than $\hat{\pr}_{p_0}$, therefore
$$
	\pr(\Psi(\Theta_{(0,1)}(\mathcal{H}))\in \Gamma_M(j))>\epsilon^{2}_1.
	$$
On the other hand, the event  $E_3$  depends on marks in the region $\mathbb{Z}\times[S-1,S]$. Note that this event has probability 
\begin{equation}\label{Eq 1}
\pr(E_{3})=\left(\frac{1-e^{-(R\lambda +1)}}{R\lambda+1}\right)^{\frac{\imath\hat{N}}{2}-\frac{\hat{N}}{2}}\left(\frac{1-e^{-(R\lambda +1)}}{R\lambda+1}\right)^{N-\frac{M\hat{N}}{2}-\frac{\hat{N}}{2}}.
\end{equation}
Therefore, since $N-\frac{M\hat{N}}{2}-\frac{\hat{N}}{2}\leq 2\alpha\hat{ K}\hat{N} - \frac{\hat{N}}{2}$, we conclude that  $\pr(E_3)$ has a positive lower bound, say $\beta$, that does not depend on $N$. Thus, by the Markov property, \eqref{claim 1 eq1} holds for $\hat{\eta}=\Tilde{p}\epsilon^{2}_1\beta$.

Finally, we observe that for $x \in \left[\frac{-M\hat{N}}{2}-\frac{\hat{N}}{2},\frac{-\imath\hat{N}}{2}+\frac{\hat{N}}{2}\right]$ the proof is analogous. 
\end{proof}

\section{Regeneration for the classical contact process}\label{Section4}
In this section, we present a result about the metastability for the classical contact process in dimension $1$ with range $R\geq1$, which will be called the \emph{regeneration} property. This property was introduced in  \cite{Mountford} for the classical contact process in dimension $2$.  Roughly, we say that a contact process regenerates if, with probability close to one, either  the process beginning with a fixed initial configuration is the empty set at a certain time $a_N$ or at this time the infected
sites are the same as for the process that begins with full occupancy. In addition, the probability goes to one when $N$ goes to infinity uniformly with respect to the initial configuration, and $a_N$ is negligible compared with the extinction time. In the following proposition, we give a precise statement of the  regeneration property for the classical contact process in dimension $1$ and $R\geq 1$.
		\begin{proposition}\label{Prop0}
		There	exist sequences $a_N$ and $b_N$ that satisfy 
			\begin{center}
				\begin{itemize}
				
					\item[(i)]There is $c>0$ such that 
					\begin{equation}\label{Prop0itemi}
					\underset{\xi_0\in{\{0,1\}}^{[1,N]}}{\sup}\pr(\xi^{\textbf{1},N}_{a_N}\neq\xi^{\xi_0,N}_{a_{N}} ;T^{\xi_0}_{N}>a_N)\leq c^{N},
					\end{equation} for $N$ large enough. 
				\item[(ii)] $\frac{b_N}{a_N}\rightarrow\infty$.
					\item[(iii)] $\underset{N\rightarrow \infty}{\lim} \pr(T^{\textbf{1}}_N<b_N)=0.$
				\end{itemize}
			\end{center}
			In particular, we have that $a_N=(\hat{K}\hat{N}M^{2}+3)N$ and $b_N=e^{\frac{c_{\infty}}{2}N}$, for a constant  $c_{\infty}>0$.
		\end{proposition}
We restrict the proof of this proposition to the case $R>1$. The idea for the classical contact process nearest neighbor ($R=1$)  is the same, the  only difference is that in this case it is used  property  \eqref{N-barrierD=1} instead of the object  $N$-\emph{barrier}.	
\begin{proof}[Proof of Proposition \ref{Prop0}]
We start by proving item $(i)$. Fix $N$ large enough such that Proposition \ref{N-barrier} holds and consider $M=M(N)$ and $S=S(N)$ given in Definition \ref{DefinitionN-barrier}, $\imath=\imath(N)$ as in \eqref{imath(m)}  and the interval $A_2$ as in \eqref{Prop 0 eq0}. Also, for $i \in \mathbb{N}$ define
\begin{align*}
    B_{i}=\{\exists\hspace{0.1cm} x \in A_2: \xi^{\xi_0,N}_{s_{i-1}+1}(x)=1 \text{ and } (x,s_{i-1}+1)\hspace{0.1cm} \text{ is }N\text{-\emph{barrier}}\},
\end{align*}
where $s_i=(S+1)i$ and $s_0=0$. 

Now, observe that by the Markov property it holds
$$
\pr(\exists\hspace{0.1cm} x \in A_2 \text{ such that } \xi^{\xi_0,N}_{s_{i-1+1}}(x)=1|T^{\xi_0}_N>s_{i-1})\geq \underset{x\in[1,N]}{\min}\pr((x,0)\rightarrow A_2\times\{1\}).
$$
Furthermore, by the definitions of $A_2$ and $\imath$ we have that the left extreme of the interval $A_2$ is smaller than $(\lfloor4 \alpha \hat{K}\hat{N} \rfloor+1)\hat{N}/2 -\hat{N}/2$. Hence,  we can choose $\tilde{\eta}>0$ independent  of  $N$ such that
 \begin{equation}\label{Prop 0 eq*}
 \underset{x\in[1,N]}{\min}\pr((x,0)\rightarrow A_2\times\{1\})\geq\tilde{\eta}.
 \end{equation}
Moreover, for  any $\xi\in \{0,1\}^{[1,N]}$ such that $\xi \neq \emptyset$ we have that
\begin{align*}
&\pr(B_k;\xi^{\textbf{1},N}_{s_{k-1}}=\xi)=\sum \limits_{\tilde{\xi}\cap A_2\neq \emptyset}\pr(B_k;\xi^{\textbf{1},N}_{s_{k-1}}=\xi;\xi^{\textbf{1},N}_{s_{k-1}+1}=\tilde{\xi})\\
&=\sum \limits_{\tilde{\xi}\cap A_2\neq \emptyset}\pr(B_k;\xi^{\textbf{1},N}_{s_{k-1}}=\xi|\xi^{\textbf{1},N}_{s_{k-1}+1}=\tilde{\xi})\pr(\xi^{\textbf{1},N}_{s_{k-1}+1}=\tilde{\xi})\\
&=\sum \limits_{\tilde{\xi}\cap A_2\neq \emptyset}\pr(B_k|\xi^{\textbf{1},N}_{s_{k-1}+1}=\tilde{\xi})\pr(\xi^{\textbf{1},N}_{s_{k-1}}=\xi|\xi^{\textbf{1},N}_{s_{k-1}+1}=\tilde{\xi})\pr(\xi^{\textbf{1},N}_{s_{k-1}+1}=\tilde{\xi})\\
&\geq \hat{\eta}\pr(\xi^{\textbf{1},N}_{s_{k-1}+1}\cap A_2 \neq \emptyset;\xi^{\textbf{1},N}_{s_{k-1}}=\xi)\\
&=\hat{\eta}\pr(\xi^{\textbf{1},N}_{s_{k-1}+1}\cap A_2 \neq \emptyset|\xi^{\textbf{1},N}_{s_{k-1}}=\xi)\pr(\xi^{\textbf{1},N}_{s_{k-1}}=\xi)\geq \hat{\eta}\tilde{\eta}\pr(\xi^{\textbf{1},N}_{s_{k-1}}=\xi),
\end{align*}
where the third equality is by the Markov property, the first inequality uses \eqref{claim 1 eq1} and the second one uses \eqref{Prop 0 eq*}. Therefore
\begin{equation}\label{Prop 0 eq**}
\pr(B_k|\xi^{\textbf{1},N}_{s_{k-1}}=\xi)\geq \hat{\eta}\tilde{\eta}
\end{equation}
 and for $k\geq1$ we have that
\begin{align}\label{Prop 0 eq1}
\begin{split}
&\pr(\underset{i\leq k}{\cap}B_{i}^{c}\cap \{T^{\xi_0}_N>s_{k-1}\})\\
&=\sum\limits_{\xi\neq \emptyset}\pr(B_k^c;\underset{i\leq k-1}{\cap}B_{i}^{c}\cap \{T^{\xi_0}_N>s_{k-2}\}|\xi^{\textbf{1},N}_{s_{k-1}}=\xi)\pr(\xi^{\textbf{1},N}_{s_{k-1}}=\xi)\\
&=\sum\limits_{\xi\neq \emptyset}\pr(B_k^c|\xi^{\textbf{1},N}_{s_{k-1}}=\xi)\pr(\underset{i\leq k-1}{\cap}B_{i}^{c}|\xi^{\textbf{1},N}_{s_{k-1}}=\xi)\pr(\xi^{\textbf{1},N}_{s_{k-1}}=\xi)\\
&\leq (1-\hat{\eta}\tilde{\eta})\pr(\underset{i\leq k-1}{\cap}B_{i}^c\cap \{T^{\xi_0}_N>s_{k-2}\}),
\end{split}
\end{align}
%
where the second equality follows by the Markov property and the inequality uses \eqref{Prop 0 eq**}. Now, using \eqref{Prop 0 eq1} recursively  we obtain that
\begin{equation}\label{eq:1}
\pr(\underset{i\leq N}{\cap}B_{i}^{c}\cap \{T^{\xi_0}_N>s_{N-1}) \leq(1-\hat{\eta}\tilde{\eta})^{N}.
\end{equation}

To conclude the proof, we set $c=1-\hat{\eta}\tilde{\eta}$, $a_N=s_{N}$ and prove the following inclusion
\begin{equation}\label{eq:2}
\{\xi^{\textbf{1},N}_{a_{N}}\neq\xi^{\xi_0,N}_{a_{N}};T^{\xi_0}_N> a_{N}\}\subset\underset{i\leq N}{\bigcap}B_i^{c} \cap \{T^{\xi_0}_N> a_N\}.
\end{equation}
To do this, first observe that the inclusion \eqref{eq:2} is equivalent to 
\begin{equation}\label{eq:2*}
\underset{i\leq N}{\bigcup}B_i \cup \{T^{\xi_0}_N\leq a_N\}\subset\{\xi^{\textbf{1},N}_{a_{N}}=\xi^{\xi_0,N}_{a_{N}}\}\cup\{T^{\xi_0}_N\leq a_{N}\}.
\end{equation}
Moreover, observe that to obtain the inclusion \eqref{eq:2*} it is  sufficient to prove that
\begin{equation}\label{Prop 0 eq2}
\underset{i\leq N}{\bigcup}B_i \cap \{T^{\xi_0}_N> a_N\}\subset\{\xi^{\textbf{1},N}_{a_{N}}=\xi^{\xi_0,N}_{a_{N}}\}.
\end{equation}
Therefore, we will prove \eqref{Prop 0 eq2}, which yields \eqref{eq:2}. Take a realization in the event on the left member of \eqref{Prop 0 eq2} and take  $i$ such that there exists $x$ satisfying $\xi^{\xi_0,N}_{s_{i-1}+1}(x)=1$ and  $(x,s_{i-1}+1)$ is an $N$\emph{-barrier}. Then, by the definition of $N$-\emph{barrier}  we have
\begin{itemize}
    \item[]$\forall$  $y \in [1,N] $ such that $[1,N]\times \{s_{i-1}+1\}\rightarrow(y,s_{i})$, we have that $(x,s_{i-1}+1) \rightarrow(y,s_{i})$ inside $[1,N]$,
\end{itemize}
which implies that $\xi^{\textbf{1},N}_{s_i}=\xi^{\xi_0,N}_{s_i}$ and consequently the processes are equal at time $a_N$. From this, we deduce \eqref{Prop 0 eq2}. Item $(i)$ now follows from \eqref{eq:1} and  \eqref{Prop 0 eq2}.

For item $(iii)$ we use the next result: there exists $c_{\infty}>0$ such that
\begin{equation}\label{T_NConvergeenPR}
\underset{N\rightarrow \infty}{\lim}\frac{1}{N}\log T^{\textbf{1}}_N=c_{\infty} \text{ in Probability}.
\end{equation}
Clearly, by \eqref{T_NConvergeenPR} if we take $b_N=e^{\frac{c_{\infty}}{2}N}$ item $(iii)$ holds. We discuss the result \eqref{T_NConvergeenPR} in Remark \ref{Remark1} below. 

 Item $(ii)$ follows immediately from the choice of $a_N$ and $b_N$.

\end{proof}

\begin{remark}\label{Remark1}
For the nearest neighbor scenario it was shown in \emph{\cite{Durrett-Liu}} that for any $\epsilon>0$ 
 \begin{equation}\label{Prop 0 eq3}
\underset{N\rightarrow \infty}{\lim} \pr\left(\frac{1}{N} \log T^{\textbf{1}}_N >c_{\infty}+ \epsilon\right)=0,
 \end{equation}
 and in \emph{\cite{Schonmman-Durrett}} it was proved that
 \begin{equation}\label{Prop 0 eq4}
\underset{N\rightarrow \infty}{\lim} \pr\left(\frac{1}{N} \log T^{\textbf{1}}_N <c_{\infty}-\epsilon\right)=0. 
\end{equation}
 Clearly, these results imply \eqref{T_NConvergeenPR} for $R=1$. The proofs of \eqref{Prop 0 eq3} and \eqref{Prop 0 eq4} use the fact that there exists $\Hat{c}>0$ such that for all $t\geq 0$
\begin{equation}\label{Fact(i)Durrett}
\pr(t<T^{[1,N]}< \infty)\leq e^{-t \Hat{c}},
\end{equation}
which was proved in \emph{\cite{Durrett-Griffeath}}. Formula \eqref{Fact(i)Durrett} is obtained by the Peierls contour argument. 

When $R>1$, we can obtain \eqref{Fact(i)Durrett} using the same argument  except that   the renormalization used in the previous case is replaced by the Mountford-Sweet renormalization. The other steps of the proof of \eqref{Prop 0 eq3} and \eqref{Prop 0 eq4} for the  nearest neighbor case are also valid when $R>1$.

\end{remark}

Once we have the regeneration property, we can get the asymptotic exponentiality for $T^{\textbf{1}}_N/\mathbb{E}(T^{\textbf{1}}_N)$,  as in Proposition $(1.2)$ of \cite{Mountford}.
\begin{corollary}
$$
\underset{N\rightarrow \infty}{\lim}\frac{T^{\textbf{1}}_N}{\mathbb{E}(T^{\textbf{1}}_N)}=E \text{ in Distribution},
$$
where $E$ has exponential distribution with rate $1$.
\end{corollary}
		\section{Metastability}\label{Metastability}
In this section, we prove Theorem \ref{TeoMetaD=1R>1}. We start by proving a proposition that will imply that the probability of the event ``both types of particles survive until time $a_{2N}$ but there is no particle of type $2$ in  $[1,N]\times[0,a_{2N}]$" is exponentially small on $N$.

Given $i\geq1 $ and $N$, define the following stopping times 
		\begin{equation}\label{S_k}
		S^{\zeta_0}_i=\inf \{t>(i-1)a_{2N}:\exists\hspace{0.1cm} y \in [1,N],\hspace{0.1cm} \zeta^{\zeta_0,N}_{t}(y)=2\},
		\end{equation}
and
		$$
		\hat{S}^{\zeta_0}_i=\inf \{t>(i-1)a_{2N}:\exists \hspace{0.1cm} x \in [-N+1,0],\hspace{0.1cm} \zeta^{\zeta_0,N}_{t}(x)=1\}.
		$$
Note that,	by the symmetry of the Harris construction,  $\hat{S}^{\zeta_0}_i$ and $S^{\zeta_0}_i$ have the same distribution. Therefore, we state the following result only for $S^{\zeta_0}_i$, but it will  also be valid for $\hat{S}^{\zeta_0}_i$.
	\begin{proposition}\label{Prop1} Consider $\mathcal{C}=\{\zeta_0 \in \{0,1,2\}^{[-N+1,N]}: \exists\hspace{0.2cm} x, y\hspace{0.2cm} \zeta_0(x)=1,\zeta_0(y)=2\}$.  Then, there exists  $c$, $0<c<1$, such that
			\begin{equation}\label{Eq 1 Prop1}\underset{\zeta_0 \in \mathcal{C}}{\sup}\,\pr(\tau^{\zeta_0}_N>2N^2a_{2N}; \exists \hspace{0.1cm}i \hspace{0.2cm} 1\leq i\leq N:S^{\zeta_0}_{i}>ia_{2N})\leq 2N^2 c^{2N},
			\end{equation}
			for all $N$ large enough.
		\end{proposition}
To prove Proposition \ref{Prop1} we will need the next lemma. 
\begin{lemma}\label{PathsLemma}
Let $A$ and $B$ be disjoint subsets of $[-N+1,N]$. Given the construction of the two-type contact process with initial configuration $\zeta^{A,B,N}_0 = \mathds{1} _{A}+2\mathds{1} _{B}$,  we have that
  $\zeta_{t}^{A, B,N}(x)=1$ if and only if there exists a path  $\gamma$  connecting $A$  with  $(x,t)$
  such that $\zeta^{A,B,N}_{s}(\gamma(s))=1$, for all  $s$, $0 \leq s \leq t$,  where $x\in [-N+1,N]$ and $t>0$.
\end{lemma}

\begin{proof}
$(\Leftarrow)$ It is clear.\

$(\Rightarrow)$
For a given realization of the Harris construction, let $m$ be the number of the marks of the Poisson processes $ \{P^{x}\}_{x \in [-N+1,N ]}$ and 
$\{P^{x\rightarrow y}\}_{\{x,y \in [-N+1,N ]: \hspace{0.2cm} 0<|x-y|\leq R\}}$ that appear before time $t$. Let $t_i$ be the time of the $i$-th mark,  and set $t_{0}=0$ and $t_{m+1}=t$. 
We now proceed  by induction on $i$, $0\leq i\leq m+1$. For $i=0$ it  is clear that the statement holds for $t_0=0$.  Suppose that the statement is valid for $i$. Then, take $y$ such that $\zeta^{A,B,N}_{t_{i+1}}(y)=1$. We must find a path $\beta$ connecting $A\times \{0\}$ with $(y,t_{i+1})$ with the desired properties. There are two possibilities:
\begin{enumerate}
\item  $\zeta^{A,B,N}_{t_{i+1}}(y)=\zeta^{A,B,N}_{t_{i}}(y)$. In this case,  by  the induction hypothesis we have that there is $\gamma$ connecting $A \times \{0\}$ with $(y,t_{i})$ such that $\zeta^{A,B,N}_{s}(\gamma(s))=1, \hspace{0.2cm}  0 \leq s \leq t_{i}$, and we define
$$\beta(s)= \left \lbrace \begin{array}{cc}
\gamma(s) & 0\leq s < t_{i},\\
y & t_{i}\leq s \leq t_{i+1}.
\end{array}\right. $$

    \item $\zeta^{A,B,N}_{t_{i+1}}(y)\neq\zeta^{A,B,N}_{t_{i}}(y)$. In this case, there is an integer  $k \in [-R, R]\setminus \{0\}$  such that $t_{i+1}\in P^{y+k\rightarrow y}$ and $\zeta^{A,B,N}_{t_{i+1}}(y)=\zeta^{A,B,N}_{t_{i}}(y+k)$. By the induction hypothesis we have that there is $\gamma$ connecting $A\times \{0\}$ with $(y+k,t_{i})$ satisfying $\zeta^{A,B,N}_{s}(\gamma(s))=1, \hspace{0.2cm} 0 \leq s \leq t_{i},$ and we define
$$\beta(s)= \left \lbrace \begin{array}{cc}
\gamma(s) & 0\leq s < t_{i}\\
y+k & t_{i}\leq s <t_{i+1}\\
y & s=t_{i+1}.
\end{array}\right. $$
\end{enumerate}
Since in each case, the path $\beta$ satisfies $$\zeta^{A,B,N}_{s}(\beta(s))=1, \hspace{0.2cm} \text{ for all }s,\, 0 \leq s \leq t_{i+1},$$
the proof of the lemma is complete.
\end{proof}
		\begin{proof}[Proof of Proposition \ref{Prop1}]
 For $k\geq 2$ we have that 
			\begin{align}\label{label}
				\begin{split}\pr(\tau^{\zeta_0}_N>Na_{2N}; S^{\zeta_0}_{k}>ka_{2N})
				&\leq \sum_{\hat{\zeta_0}\in \mathcal{C}} \pr(\tau^{\hat{\zeta}_0}_N>a_{2N};S^{\hat{\zeta}_0}_1>a_{2N})\pr(C_k=\hat{\zeta}_0)\\&
				= \underset{\hat{\zeta}_0 \in \mathcal{C}}{\sup}\, \pr(\tau^{\hat{\zeta}_0}_N>a_{2N}, S^{\hat{\zeta}_0}_1>a_{2N}).
				\end{split}
			\end{align}
%
Thus, to obtain  \eqref{Eq 1 Prop1} it is enough to prove
		\begin{equation}\label{Eq 2 Prop1}
		    \underset{\hat{\zeta}_0 \in \mathcal{C}}{\sup}\, \pr(\tau^{\hat{\zeta}_0}_N>a_{2N}; S^{\hat{\zeta}_0}_1>a_{2N})\leq c^{2N},
		\end{equation}
			for some $c$, $0<c<1$.
			
 To simplify notation, only throughout the proof,   we let  $\xi^{2N}_{a_{2N}}$ stand for the classical contact process restricted to the interval $[-N+1,N]$.   
 
 Now, we observe that the event inside the probability in \eqref{Eq 2 Prop1} can be written as 		
\begin{align*}
\begin{split}
\{\tau^{\hat{\zeta}_0}_N>a_{2N}; S^{\hat{\zeta}_0}_1>a_{2N}\}=&\{\tau^{\hat{\zeta}_0}_{N}>a_{2N};S^{\hat{\zeta}_0}_1>a_{2N};\xi^{B,2N}_{a_{2N}}=\xi^{\textbf{1},2N}_{a_{2N}}\}\\&\cup\{\tau^{\hat{\zeta}_0}_{N}>a_{2N};S^{\hat{\zeta}_0}_1>a_{2N};\xi^{B,2N}_{a_{2N}}\neq\xi^{\textbf{1},2N}_{a_{2N}}\}.
\end{split}
\end{align*}
Next, we set $B=B(\hat{\zeta}_0)=\{x: \hat{\zeta}_0(x)=1\}$ and we claim that
			\begin{equation}\label{Eq 4 Prop1}
				\{\tau^{\hat{\zeta}_0}_{N}>a_{2N};S^{\hat{\zeta}_0}_1>a_{2N};\xi^{B,2N}_{a_{2N}}=\xi^{\textbf{1},2N}_{a_{2N}}\}=\emptyset.
			\end{equation}
Observe that this claim implies that
			\begin{align}\label{Eq 3 Prop1}
			\begin{split}
			    	\{\tau^{\hat{\zeta}_0}_N>a_{2N}; S^{\hat{\zeta}_0}_1>a_{2N}\}=&\{\tau^{\hat{\zeta}_0}_{N}>a_{2N};S^{\hat{\zeta}_0}_1>a_{2N};\xi^{B,2N}_{a_{2N}}\neq\xi^{\textbf{1},2N}_{a_{2N}}\}\\&\subset \{T^{B}_{2N}>a_{2N};\xi^{B,2N}_{a_{2N}}\neq\xi^{\textbf{1},2N}_{a_{2N}}\}.
			    	\end{split}
			\end{align}
Hence, \eqref{Eq 2 Prop1} follows from \eqref{Eq 3 Prop1} and Proposition \ref{Prop0} item $(i)$.
			
 Thus, the proof is completed by showing \eqref{Eq 4 Prop1}. For this purpose, it is enough to show that  every realization in $\{S^{\hat{\zeta}_0}_1>a_{2N};\xi^{B,2N}_{a_{2N}}=\xi^{\textbf{1},2N}_{a_{2N}}\}$ is in $\{\tau^{\hat{\zeta}_0}_{N}\leq a_{2N}\}$.  Take $x \in \xi^{\textbf{1},2N}_{a_{2N}}$ and let $\gamma$ be a path  connecting $B\times\{0\}$ with $(x,a_{2N})$. For $\gamma$ we define $s^{*}$ by 
			$$s^{*}=\inf \{t: 0<t\leq a_{2N} ,\zeta^{\hat{\zeta}_0,2N}_t(\gamma(t))=2\},$$
with the usual convention that $\inf\{\emptyset\}=\infty$.

Suppose that $s^{*}<\infty$.  Since $S^{\hat{\zeta}_0}_1>a_{2N}$ and $\zeta^{\hat{\zeta}_0,2N}_{s^{*}}(\gamma(s^{*}))=2$, we obtain that $\gamma(s^{*})\in [-N+1,0]$. However, by the definition of $s^{*}$, we have that  $\zeta^{\hat{\zeta}_0,2N}_t(\gamma(t))=1$ for all $t<s^{*}$, which implies that $\gamma$ restricted to $[0,s^{*}]$ is a path of particles $1$ that infects the site $\gamma(s^{*})$  at time $s^{*}$. Since the particles of type $1$ have priority in $[-N+1,0]$, we get $\zeta^{\hat{\zeta}_0,2N}_{s^{*}}(\gamma(s^{*}))=1$. This is a contradiction and we conclude that $s^{*}=\infty$, which means $\zeta^{\hat{\zeta}_0,2N}_t(\gamma(t))=1$ for all $0\leq t\leq a_{2N}$. Therefore, $\zeta^{\hat{\zeta}_0,2N}_{a_{2N}}(x)=1$ for all $x\in \xi^{\textbf{1},2N}_{a_{2N}}$.	
		\end{proof}	
		\subsection{Proof of Theorem \ref{TeoMetaD=1R>1}}
		 We are now ready to state the regeneration property for the process $\{\zeta^{\textbf{1},\textbf{2},N}_t\}$. The main idea is to prove that if the two types of particles survive for a given time polynomial in $N$, then  outside an event with  exponentially    small probability we can find two \emph{barriers} at the same time, one in $[-N+1,0]$ and the other in $[1,N]$, such that the first one is infected by a particle of type $1$ and the second by a particle of type $2$. Basically, we combine the idea of the proof of Proposition \ref{Prop0} with Proposition \ref{Prop1} to obtain the following:
		\begin{proposition}\label{Prop2}
			There are  sequences $c_N$ and $d_N$ that satisfy 
			\begin{center}
				\begin{itemize}
				\item[(i)] There exists $\nu$, $0<\nu<1$, such that for $N$ large enough 
\begin{equation}\label{Eqnu}
\underset{\zeta \in \mathcal{C}}{\sup}\,\pr(\zeta^{\textbf{1},\textbf{2},N}_{c_N}\neq \zeta^{\zeta,N}_{c_N} ;\tau^{\zeta}_{N}\geq c_N)\leq \nu^{N},
\end{equation}
				\item[(ii)] $\frac{d_N}{c_N}\rightarrow\infty,$
					
					\item[(iii)] $\underset{N\rightarrow \infty}{\lim} \pr(\tau^{\textbf{1},\textbf{2}}_N<d_N)=0,$	
				\end{itemize}
			\end{center}
			where $\mathcal{C}$ has been defined in Proposition \ref{Prop1}. In particular, we have that $c_N=2N^2 a_{2N}$ and $d_N=b_N$.
		\end{proposition}
		\begin{proof}
Observe that $\tau^{\textbf{1},\textbf{2}}_N$ is stochastically larger than the minimum of two independent variables with the same law of  $T^{\textbf{1}}_N$. Hence, taking $d_N=b_N$, defined in Proposition \ref{Prop0},  item $(iii)$ is immediate. 

Now,  we take $N$ large enough as in Proposition \ref{N-barrier} and Proposition \ref{Prop1},  $M=M(N)$ and $S=S(N)$ as in Definition \ref{DefinitionN-barrier},  $\imath=\imath(N)$ in \eqref{imath(m)}, the interval $A_2$ in \eqref{Prop 0 eq0} and we define $A_4=\left[\frac{-M\hat{N}}{2}-\frac{\hat{N}}{2},\frac{-\imath\hat{N}}{2}+\frac{\hat{N}}{2}\right]$. Also, for $k \in \{1, \dots ,N\}$, let $s_k=k2N  a_{2N}$,
$$
\Lambda^{\zeta_0}_{N,k}= \left\{\begin{array}{l}\exists\hspace{0.1cm} y_k  \in A_2, z_{k} \in A_4:    \zeta^{\zeta_0,N}_{s_k}(y_k)=2, \zeta^{\zeta_0,N}_{s_k}(z_k)=1,\\  \text{ and } (y_k,s_k),(z_{k},s_k), \text{ are } N\text{-\emph{barriers}} \end{array}\right\},
$$
and
$$
\Lambda_{N,k}=\left \lbrace\begin{array}{l}\exists\hspace{0.1cm} \hat{y}_{k}\in A_2, \hat{z}_{k} \in A_4: \zeta^{\textbf{1},\textbf{2},N}_{s_k}(\hat{y}_k)=2,\zeta^{\textbf{1},\textbf{2},N}_{s_k}(\hat{z}_k)=1,\\
\text{ and }(\hat{y}_k,s_k),(\hat{z}_k,s_k) \text{ are } N\text{-\emph{barriers}}\end{array}\right \rbrace.
$$
To obtain item $(i)$, we first prove the following inclusion
\begin{align}\label{Prop 2 eq1}
\begin{split}
&\{\tau^{\zeta_0}_N>2N^2 a_{2N};\tau^{\textbf{1},\textbf{2}}_N>2N^2 a_{2N}\}\cap \bigcup^{N}_{k=1} \Lambda^{\zeta_0}_{N,k}\cap \Lambda_{N,k}\subset\\& 
 \{\tau^{\zeta_0}_N>2N^2 a_{2N};\tau^{\textbf{1},\textbf{2}}_N>2N^2 a_{2N};
\zeta^{\zeta_0,N}_{2N^2 a_{2N}}=\zeta^{\textbf{1},\textbf{2},N}_{2N^2 a_{2N}}\}.
\end{split}
\end{align}Fix a realization in $ \{\tau^{\zeta_0}_N>2N^2 a_{2N};\tau^{\textbf{1},\textbf{2}}_N>2N^2 a_{2N}\}\cap  \Lambda^{\zeta_0}_{N,k}\cap \Lambda_{N,k}$.  By the definition of $N$-\emph{barrier}, we have that the points $(y_k,s_k)$, $(\hat{y}_k, s_k)$, $(z_k,s_k)$ and $(\hat{z}_k, s_k)$ satisfy:

\begin{itemize}
    \item[] If $y\in [1,N]$  and $[1,N]\times\{s_k\}\rightarrow(y,s_k+S)$, we have $(y_k,s_k)\rightarrow(y,s_k+S)$ inside $[1,N]$  and  $(\hat{y}_k, s_k)\rightarrow(y,s_k+S)$ inside $[1,N]$.
    \item[] If $z\in [-N+1,0]$  and $[-N+1,0]\times\{s_k\}\rightarrow(z,s_k+S)$, we have  $(z_k,s_k)\rightarrow(z,s_k+S)$ inside $[-N+1,0]$  and  $(\hat{z}_k, s_k)\rightarrow(z,s_k+S)$ inside $[-N+1,0]$.
\end{itemize}
Therefore, we  can conclude that
$$\zeta^{\zeta_0,N}_{s_k+S}= \zeta^{\textbf{1},\textbf{2} ,N}_{s_k+S}=2\eta^{[1,N]}_{s_{k},S}\text{ in } [1,N]$$
and
$$\zeta^{\zeta_0,N}_{s_k+S}= \zeta^{\textbf{1},\textbf{2},N }_{s_k+S}=\eta^{[-N+1,0]}_{s_{k},S}\text{ in } [-N+1,0].$$
Consequently, we obtain that $\zeta^{\zeta_0,N}_{2N^2 a_{2N}}=\zeta^{\textbf{1},\textbf{2},N}_{2N^2 a_{2N}}$, which proves \eqref{Prop 2 eq1}.

Next, we choose  $c_N=2N^2 a_{2N}$ and we prove that there exists $0<\nu<1$ such that
\begin{small}
\begin{equation}\label{Prop 2 eq2}
\pr\left(\{\tau^{\zeta_0}_N>2N^2a_{2N};\tau^{\textbf{1},\textbf{2}}_N>2N^2a_{2N}\} \cap\bigcap_{k=1}^{N}(\Lambda^{\zeta_0}_{N,k}\cap \Lambda_{N,k})^{c}\right)\leq {\nu}^{N}
\end{equation}
\end{small}
for all $\zeta_0 \in \mathcal{C}$. To do this, observe that by simple manipulations we have
\begin{small}\begin{align}\label{Prop 2 eq3}
\begin{split}
 &\pr\left(\{\tau^{\zeta_0}_N>2N^2 a_{2N};\tau^{\textbf{1},\textbf{2}}_N>2N^2 a_{2N}\} \cap\bigcap_{k=1}^{N}(\Lambda^{\zeta_0}_{N,k}\cap\Lambda_{N,k})^{c}\right)\\
 &=\pr\left(\{\tau^{\zeta_0}_N>2N^2 a_{2N};\tau^{\textbf{1},\textbf{2}}_N>2N^2 a_{2N}\} \cap\bigcap_{k=1}^{N}((\Lambda^{\zeta_0}_{N,k})^{c}\cup\Lambda_{N,k}^{c})\right)\\
 & \leq\pr\left(\{\tau^{\zeta_0}_N>2N^2 a_{2N};\tau^{\textbf{1},\textbf{2}}_N>2N^2 a_{2N}\} \cap\bigcap_{k=1}^{N}(\Lambda^{\zeta_0}_{N,k})^{c})\right)\\
& \hspace{0.1cm}+\pr\left(\{\tau^{\zeta_0}_N>2N^2 a_{2N};\tau^{\textbf{1},\textbf{2}}_N>2N^2 a_{2N}\} \cap\bigcap_{k=1}^{N}\Lambda_{N,k}^{c}\right).
   \end{split}
\end{align}\end{small}
In order to estimate the last two terms in \eqref{Prop 2 eq3}, observe that each of them is less than 
$\underset{\zeta_0\in\mathcal{C}}{\sup\,}\pr\left(\{\tau^{\zeta_0}_N>2N^2 a_{2N}\} \cap\bigcap_{k=1}^{N}(\Lambda^{\zeta_0}_{N,k})^{c})\right),$ and furthermore we have that
\begin{align}\label{Prop 2 eq*}
\begin{split}
 &\underset{\zeta_0\in\mathcal{C}}{\sup\,}\pr\left(\{\tau^{\zeta_0}_N>2N^2 a_{2N}\} \cap\bigcap_{k=1}^{N}(\Lambda^{\zeta_0}_{N,k})^{c})\right)\\
 &\leq\underset{\zeta_0\in\mathcal{C}}{\sup\,}\pr\left( \{\tau^{\zeta_0}_N>2N^2 a_{2N}\}\cap \bigcap_{k=1}^{N}\left \lbrace\begin{array}{l}   \forall y\hspace{0.1cm} \in A_2 \hspace{0.1cm}\zeta^{\zeta_0,N}_ {s_k}(y)\neq2\text{ or }(y,s_k)\\ \text{ is not an } N\text{-\emph{barrier}}\end{array}\right \rbrace\right)\\
 &\hspace{0.1cm}+\underset{\zeta_0\in\mathcal{C}}{\sup\,}\pr\left( \{\tau^{\zeta_0}_N>2N^2 a_{2N}\}\cap \bigcap_{k=1}^{N}\left \lbrace\begin{array}{l}   \forall z\hspace{0.1cm} \in A_4 \hspace{0.1cm}\zeta^{\zeta_0,N}_ {s_k}(z)\neq 1\text{ or }(z,s_k)\\ \text{ is not an } N\text{-\emph{barrier}}\end{array}\right \rbrace\right)\\
 &=2\underset{\zeta_0\in\mathcal{C}}{\sup\,}\pr\left( \{\tau^{\zeta_0}_N>2N^2 a_{2N}\}\cap \bigcap_{k=1}^{N}\left \lbrace\begin{array}{l}   \forall y\hspace{0.1cm} \in A_2 \hspace{0.1cm}\zeta^{\zeta_0,N}_ {s_k}(y)\neq2\text{ or }(y,s_k)\\ \text{ is not an } N\text{-\emph{barrier}}\end{array}\right \rbrace\right),
 \end{split}
\end{align}
where the first inequality follows by the definition of the event $\Lambda^{\zeta_0}_{N,k}$ and the last equality follows by the symmetry of the Harris construction. Then, replacing \eqref{Prop 2 eq*} in \eqref{Prop 2 eq3}, we obtain that the probability in \eqref{Prop 2 eq2} is smaller than 
\begin{equation}\label{eq4'}
4\,\underset{\zeta_0\in\mathcal{C}}{\sup\,}\pr\left(\{\tau^{\zeta_0}_N>2N^2 a_{2N}\}\cap\bigcap_{k=1}^{N}\left \lbrace\begin{array}{l}   \forall y \hspace{0.1cm}\in A_2 \hspace{0.1cm}\zeta^{\zeta_0,N}_ {s_k}(y)\neq2\text{ or }(y,s_k)\\ \text{ is not an } N\text{-\emph{barrier}}\end{array}\right \rbrace\right).
\end{equation}

Now, we will estimate the probability in \eqref{eq4'}. We use  Proposition \ref{Prop1} for $i=2kN-1$ and we have that 
\begin{small}\begin{align}\label{Prop 2 eq4*}
\begin{split}
&\underset{\zeta_0\in\mathcal{C}}{\sup\,}\pr\left(\{\tau^{\zeta_0}_N>2N^2 a_{2N}\}\cap\bigcap_{k=1}^{N}\left \lbrace\begin{array}{l}   \forall y \hspace{0.1cm}\in A_2 \hspace{0.1cm}\zeta^{\zeta_0,N}_ {s_k}(y)\neq2\text{ or }(y,s_k)\\ \text{ is not an } N\text{-\emph{barrier}}\end{array}\right \rbrace\right)\leq 2N^2c^{2N} +\\&
   \underset{\zeta_0\in\mathcal{C}}{\sup\,}\pr\left(\{\tau^{\zeta_0}_N>2N^2 a_{2N}\}\cap\bigcap_{k=1}^{N}\left \lbrace\begin{array}{l} S^{\zeta_0}_{2kN-1}\leq s_k-a_{2N} \hspace{0.1cm} \forall y \hspace{0.1cm}\in A_2 \hspace{0.1cm}\zeta^{\zeta_0,N}_ {s_k}(y)\neq2\\\text{ or }(y,s_k) \text{ is not an } N\text{-\emph{barrier}}\end{array}\right \rbrace\right).\end{split}\end{align}\end{small}Thus, it is enough to prove that the last term in the inequality \eqref{Prop 2 eq4*} is exponentially small in $N$. To show this, we first observe that the last term in \eqref{Prop 2 eq4*} is smaller than\begin{small}\begin{align}\label{Prop 2 eq4}
\begin{split}
I=&\,\underset{\zeta_0\in\mathcal{C}}{\sup\,}\pr\left(\{\tau^{\zeta_0}_N>2N^2 a_{2N}\}\cap\bigcap_{k=1}^{N}\left \lbrace S^{\zeta_0}_{2kN-1}\leq s_k-a_{2N};\hspace{0.1cm}  \forall y \hspace{0.1cm}\in A_2 \hspace{0.1cm}\zeta^{\zeta_0,N}_ {s_k}(y)\neq2\right \rbrace\right)\\& 
   +\underset{\zeta_0\in\mathcal{C}}{\sup\,}\pr\left(\{\tau^{\zeta_0}_N>2N^2 a_{2N}\}\cap\bigcap_{k=1}^{N}\left \lbrace\begin{array}{l} \hspace{0.1cm} \exists\hspace{0.1cm} y \hspace{0.1cm}\in A_2\hspace{0.1cm} \zeta^{\zeta_0,N}_{s_k}(y)=2\text{ but}\\(y,s_k) \text{ is not an } N\text{-\emph{barrier}}\end{array}\right \rbrace\right).
   \end{split}
\end{align}
\end{small}
Next, we estimate the first term in \eqref{Prop 2 eq4}. We define
$$
D_k=\{x:\zeta^{\zeta_0,N}_{S_{2kN-1}}(x)=2\} \cap[1,N],
$$
and the event
\begin{align*}
C_k=\{\{D_k\}\times\{S^{\zeta_0}_{2kN-1}\}\nrightarrow A_2 \times\{S^{\zeta_0}_{2kN-1}+2a_{2N}\}\text{ inside }[1,N]\}.
\end{align*}By the priority of particles of type $2$ in $[1,N]$, we have that
\begin{align*}
&\{\tau^{\zeta_0}_N>2N^2 a_{2N}\}\cap\bigcap_{k=1}^{N}\left \lbrace S^{\zeta_0}_{2kN-1}\leq s_k-a_{2N};\hspace{0.1cm}  \forall y \hspace{0.1cm}\in A_2 \hspace{0.1cm}\zeta^{\zeta_0,N}_{s_k}(y)\neq2\right \rbrace \\&\subset\{\tau^{\zeta_0}_N>2N^2 a_{2N}\}\cap\bigcap_{k=1}^{N}\{ S^{\zeta_0}_{2kN-1}\leq s_k-a_{2N}; C_k\}.
\end{align*}
\begin{claim}\label{claim2}There exists $\beta>0$ such that for all $N$ large enough
$$
\pr(C^{c}_k |S^{\zeta_0}_{2kN-1}<s_k-a_{2N})>\beta.
$$
\end{claim}
\begin{proof}[Proof of Claim \ref{claim2}]
Fix $\xi \in \{0,1\}^{[1,N]}$ and $\xi \neq \emptyset$, then we have
\begin{align}\label{Claim1Eq1}
\begin{split}
\pr(\xi^{\xi_0,N}_{2a_{2N}}\cap A_2\neq \emptyset)&=\pr(T^{\xi_0}_N>2a_{2N};\xi^{\xi_0,N}_{2a_{2N}}\cap A_2\neq \emptyset)\\&=\pr(T^{\xi_0}_{N}>2a_{2N})-\pr(T^{\xi_0}_N>2a_{2N};\xi^{\xi_0,N}_{2a_{2N}}\cap A_2=\emptyset).
\end{split}
\end{align}
Using a Peirels contour argument for the oriented $k$-dependent system with small closure $\Psi$, defined in Section \ref{Section3}, it is possible to prove that there exist $\beta>0$ and a sequence $f_N$ linear in $N$ such that
\begin{equation}\label{Claim1Eq2'}
\underset{x \in [1,N]}{\inf}\pr(T^{\{x\}}_{N}\geq e^{f_N})>2\beta,
\end{equation}
for $N$ large enough (see \cite{Mountford} Fact $(2.2)$). Since $a_{2N}$ is of order $N^{3}$, the formula in \eqref{Claim1Eq2'} implies that
\begin{equation}\label{Delta}
\pr(T^{\xi_0}_{N}>2a_{2N})\geq 2 \beta.
\end{equation}

Now, we prove that the last term in \eqref{Claim1Eq1} goes to zero as $N$ goes to infinity. Observe that
\begin{align}\label{Claim1Eq2}
\begin{split}
\pr(T^{\xi_0}_N>2a_{2N};\xi^{\xi_0,N}_{2a_{2N}}\cap A_2=\emptyset)
&\leq \pr(T^{\xi_0}_N>2a_{2N};\xi^{\xi_0,N}_{2a_{2N}}\neq\xi^{\textbf{1},N}_{2a_{2N}})\\
&\hspace{0.1cm}+\pr(T^{\xi_0}_N>2a_{2N};\xi^{\xi_0,N}_{2a_{2N}}=\xi^{\textbf{1},N}_{2a_{2N}};\xi^{\xi_0,N}_{2a_{2N}}\cap A_2=\emptyset)\\
&\leq c^{N}+\pr(\xi^{\textbf{1},N}_{2a_{2N}}\cap A_2=\emptyset),
\end{split}
\end{align}
where the second inequality follows by item $(i)$ of Proposition \ref{Prop0}. Using the duality of the classical contact process  we have
$$
\pr (\xi^{\textbf{1},N}_{2a_{2N}}\cap A_2\neq\emptyset)=\pr(\xi^{A_2,N}_{2a_{2N}}\neq \emptyset).
$$
Observe that the length of $A_2$ is at least $l_N=N-2\alpha \hat{K} \hat{N} -\left\lfloor 4\alpha \hat{K}\hat{N} \right\rfloor-1$, then we obtain
$$
\pr (\xi^{\textbf{1},N}_{2a_{2N}}\cap A_2=\emptyset)=\pr(\xi^{A_2,N}_{2a_{2N}}= \emptyset)\leq \pr(\xi^{\textbf{1},l_N}_{2a_{2N}}= \emptyset).
$$
From  item $(iii)$ of Proposition \ref{Prop0} and the fact that $l_N$ is linear in $N$,  it follows that
$$
\underset{N\rightarrow \infty}{\lim}\pr(\xi^{\textbf{1},l_N}_{2a_{2N}}=\emptyset)=\underset{N\rightarrow \infty}{\lim}\pr(T^{\textbf{1}}_{l_N}\leq 2 a_N)=0.
$$
Thus, for $N$ large enough we have 
$$
\pr(T^{\xi_0}_N>2a_{2N};\xi^{N}_{2a_{2N}}\cap A_2=\emptyset)\leq \beta.$$
Moreover, by \eqref{Delta} and \eqref{Claim1Eq2}, we obtain
$$
\pr(\xi^{\xi_0,N}_{2a_{2N}}\cap A_2\neq \emptyset)\geq \beta,
$$
for all $\xi_0 \in \{0,1\}^{[1,N]}$, $\xi_0 \neq \emptyset$. Thus, by the strong Markov property we obtain the claim.
\end{proof}
Now, we return to the first term in \eqref{Prop 2 eq4}. Since $S^{\zeta_0}_{2kN-1}$ is larger than $2(kN-2)a_{2N}$, given the information until this time,  the event $\{S_{2kN-1}\leq s_k -a_{2N};C_k\}$ involves information between the times $(2kN-2)a_{2N}$ and $(2(k+1)N-2)a_{2N}$. Therefore, by the strong Markov property  and Claim \ref{claim2} we conclude that
\begin{align*}
 &\pr(C_j|\{S^{\zeta_0}_{2jN-1}\leq s_j-a_{2N}\}\cap\underset{1\leq k\leq j-1}{\cap}\{ S^{\zeta_0}_{2kN-1}\leq s_k-a_{2N};C_k\} )\\&=\pr(C_j|\{S^{\zeta_0}_{2jN-1}\leq s_j-a_{2N}\})\leq 1-\beta.
\end{align*} Thus, for all $j\geq1$ we have that
\begin{align}\label{Prop 2 eq 5*}
\begin{split}
&\pr(\underset{1\leq k\leq j}{\cap}\{ S^{\zeta_0}_{2kN-1}\leq s_k-a_{2N};C_k\})\\&
\hspace{3cm}\leq{(1-\beta)}\pr(\underset{1\leq k\leq j-1}{\cap}\{ S^{\zeta_0}_{2kN-1}\leq s_k-a_{2N};C_k\}).
\end{split}
\end{align}
Then, using \eqref{Prop 2 eq 5*} recursively  we obtain that
\begin{equation}\label{Prop 2 eq5**}
\pr(\underset{1\leq k\leq N}{\cap}\{ S^{\zeta_0}_{2kN-1}\leq s_k-a_{2N};C_k\})\leq (1-\beta)^{N}.
\end{equation}

Next, we analyze  the second term in  \eqref{Prop 2 eq4}. From  the fact that  $S$ is of order $N^2$  and $a_{2N}$ is of order $N^3$, we get that for $N$ large enough it holds
\begin{align*}
s_{k}-2a_{2N}&=(2kN-2)a_{2N}\leq s_k+S\leq(2(k+1)N-2)a_{2N}=s_{k+1}-a_{2N}.
\end{align*}
From these relations, we have that the $k$-th event in the intersection inside the probability in the second term of \eqref{Prop 2 eq4} involves information within the interval of time $[s_k-a_{2N},s_{k+1}-a_{2N}]$.  Hence, the  Markov property and Proposition \ref{N-barrier} imply that this probability is less than $(1-\hat{\eta})^{N}$. Thus, combining this last comment with \eqref{Prop 2 eq5**} we obtain the desired bound for \eqref{Prop 2 eq4}, specifically
\begin{equation}\label{Prop 2 eqL}
I\leq(1-\beta)^{N}+(1-\hat{\eta})^{N}.
\end{equation}

Finally,  we combine the inequality in \eqref{Prop 2 eqL} with  \eqref{Prop 2 eq4*}, \eqref{Prop 2 eq*},  \eqref{Prop 2 eq3} and select  N large enough such that
 $$4\left(2N^2c^{2N}+ (1-\beta)^{N} + (1-\hat{\eta})^{N}\right)\leq  (2 \max\{c^2, \beta, 1-\hat{\eta}\})^{N},$$ 
to obtain \eqref{Prop 2 eq2} for $\nu= 2 \max\{c^2,\beta, 1-\hat{\eta}\}$. Therefore, item $(i)$ is proved.

Item $(ii)$ follows immediately from the choice of $c_N$ and $d_N$.
\end{proof}


\begin{proof}[Proof of Theorem \ref{TeoMetaD=1R>1}]
Let $\beta_N$ as in the statement of  Theorem $1$. We will prove that
\begin{equation}\label{Teo1Eq*}
\underset{N\rightarrow \infty}{\lim}\left|\pr(\tau^{\textbf{1},\textbf{2}}_N>\beta_N(t+s))-\pr(\tau^{\textbf{1},\textbf{2}}_{N}>\beta_N t)\pr(\tau^{\textbf{1},\textbf{2}}_{N}>\beta_N s) \right|=0,
\end{equation}
which by the definition of $\beta_{N}$ will imply
$$\underset{N\rightarrow \infty}{\lim}\pr(\tau^{\textbf{1},\textbf{2}}_N\geq\beta_N t)=e^{-t}.$$

 To obtain the limit \eqref{Teo1Eq*}, we  prove that there exist two positive sequences $h_N$ and $h'_N$, both converging to zero when $N$ goes to infinity, such that 
\begin{equation}\label{Teo1Eq**}
\pr(\tau^{\textbf{1},\textbf{2}}_{N}>\beta_N t)\pr(\tau^{\textbf{1},\textbf{2}}_{N}>\beta_N s)-h_N\leq \pr(\tau^{\textbf{1},\textbf{2}}_N>\beta_N(t+s))
\end{equation}
and
 \begin{equation}\label{Teo1Eq***}
\pr(\tau^{\textbf{1},\textbf{2}}_N>\beta_N(t+s))\leq \pr(\tau^{\textbf{1},\textbf{2}}_{N}>\beta_N t)\pr(\tau^{\textbf{1},\textbf{2}}_{N}>\beta_N s)+h'_N.
\end{equation}
We begin by proving equation \eqref{Teo1Eq**}.  First, we observe that for all  positive $t$ and $s$  we have that
\begin{align}\label{Teo1Eq1}
\begin{split}
&\{\tau^{\textbf{1},\textbf{2}}_{N}>\beta_{N}s;\tau^{\textbf{1},\textbf{2},\beta_N s}_{N}>\beta_Nt;\zeta^{\textbf{1},\textbf{2},N}_{\beta_{N}s}=\zeta^{\textbf{1},\textbf{2},N}_{\beta_{N}s-c_{N},c_{N}};\zeta^{\textbf{1},\textbf{2},N}_{\beta_{N}s-c_{N},2c_{N}}=\zeta^{\textbf{1},\textbf{2},N}_{\beta_N s, c_{N}}\} \\
&\subset \{\tau^{\textbf{1},\textbf{2}}_N>\beta_N(t+s)\},
\end{split}
\end{align}
where $\zeta^{\textbf{1},\textbf{2},N}_{t,\cdot}$ and $\tau^{\textbf{1},\textbf{2},t}_{N}$ refer to the two-type  contact process defined in the restriction of the Harris construction to  $\mathbb{Z}\times[t,\infty)$. In the case $t=0$, $\zeta^{\textbf{1},\textbf{2},N}_{0,\cdot}$ is the classical contact process and we omit the subscript $0$. Observe that $\beta_N s-c_N\geq 0$ for all $s>0$ since by the definition of $\beta_N$ and $c_N$ and items $(ii)$ and $(iii)$ of Proposition \ref{Prop2} we have that $\beta_N /c_N$ converges to infinity  as $N$ goes to infinity. By formula \eqref{Teo1Eq1},  we obtain that
\begin{figure}
	\centering
		  \begin{overpic}[scale=0.4,unit=1mm]{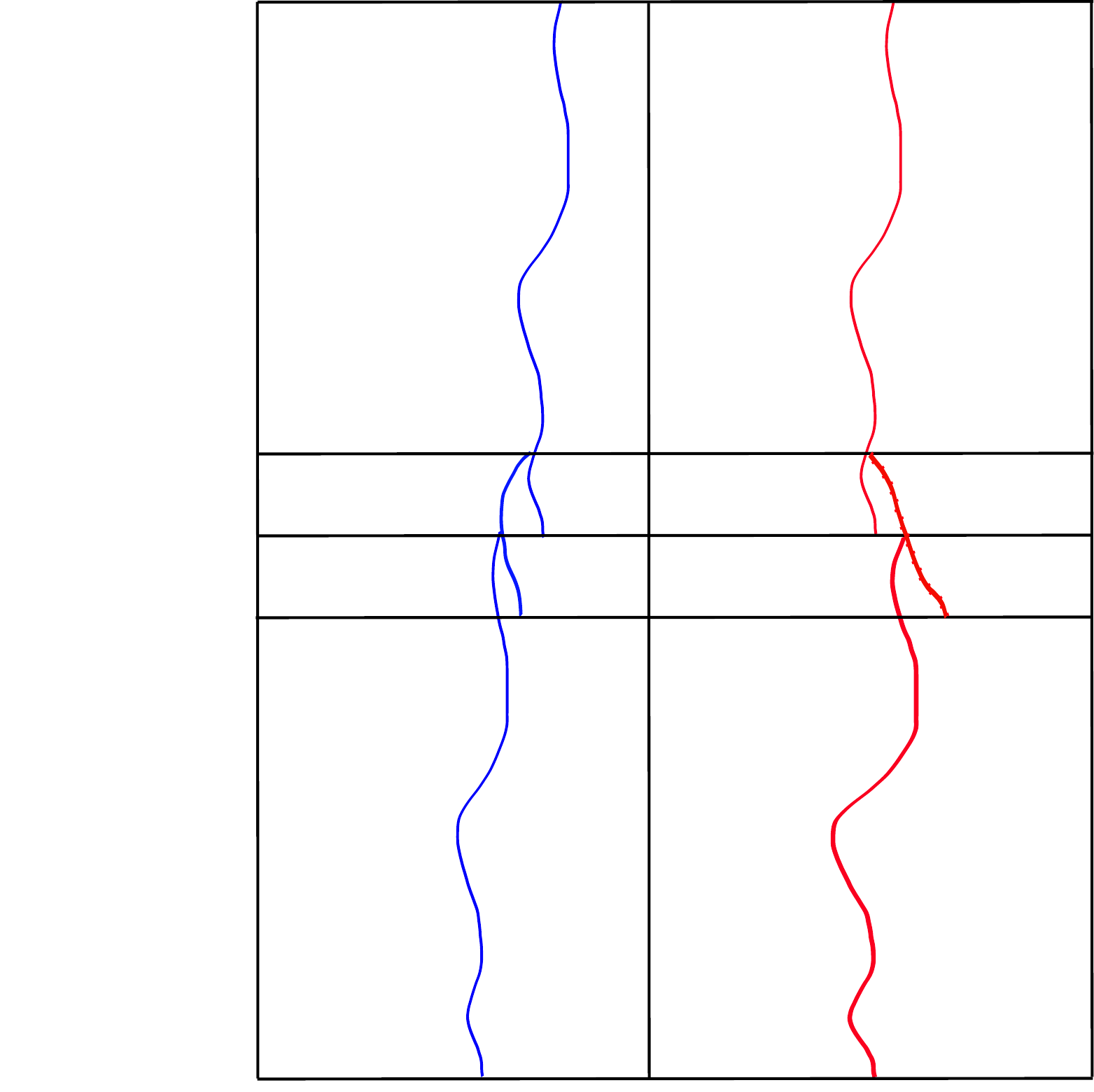}
		  \put(1,62){\parbox{0.4\linewidth}{%
  \begin{tiny}
  $\beta_N(t+s)$
  \end{tiny}}}
  \put(1,31){\parbox{0.4\linewidth}{%
  \begin{tiny}
  $\beta_N s$
  \end{tiny}}}
  \put(1,26){\parbox{0.4\linewidth}{%
  \begin{tiny}
  $\beta_N s-c_N$
  \end{tiny}}}
  \put(1,36){\parbox{0.4\linewidth}{%
  \begin{tiny}
  $\beta_N s+c_N$
  \end{tiny}}}
\end{overpic}
   \caption{A graphic representation of the inclusion \eqref{Teo1Eq1}. Blue paths represent paths of particles $1$ and red paths represent particles of type $2$.}\label{FigureMetastability}
\end{figure}
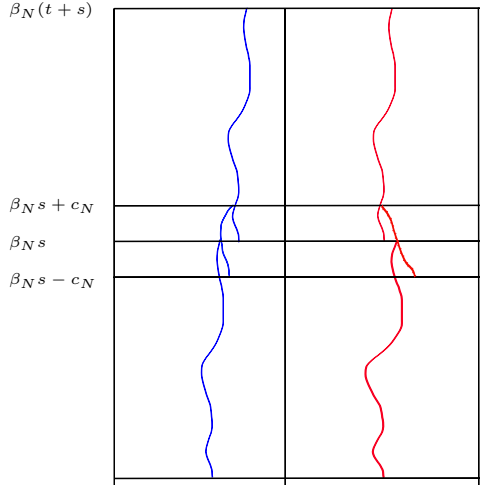
\begin{align}\label{Teo1Eq2}
\begin{split}
&\pr(\tau^{\textbf{1},\textbf{2}}_{N}>\beta_{N}s;\tau^{\textbf{1},\textbf{2},\beta_N s}_{N}>\beta_Nt;\zeta^{\textbf{1},\textbf{2},N}_{\beta_{N}s}=\zeta^{\textbf{1},\textbf{2},N}_{\beta_{N}s-c_{N},c_{N}};\zeta^{\textbf{1},\textbf{2},N}_{\beta_{N}s-c_{N},2c_{N}}=\zeta^{\textbf{1},\textbf{2},N}_{\beta_N s, c_{N}})\\&\leq \pr(\tau^{\textbf{1},\textbf{2}}_N>\beta_N(t+s)).
\end{split}
\end{align}
Now, we  choose $h_N$ as
\begin{align*}
&h_N=\pr(\tau^{\textbf{1},\textbf{2}}_{N}>\beta_{N}s)\pr(\tau^{\textbf{1},\textbf{2}}_{N}>\beta_Nt)\\&\hspace{1cm}-\pr(\tau^{\textbf{1},\textbf{2}}_{N}>\beta_{N}s;\tau^{\textbf{1},\textbf{2},\beta_N s}_{N}>\beta_Nt;\zeta^{\textbf{1},\textbf{2},N}_{\beta_{N}s}=\zeta^{\textbf{1},\textbf{2},N}_{\beta_{N}s-c_{N},c_{N}};\zeta^{\textbf{1},\textbf{2},N}_{\beta_{N}s-c_{N},2c_{N}}=\zeta^{\textbf{1},\textbf{2},N}_{\beta_N s, c_{N}}).
\end{align*}
Also, we observe that the Markov property implies that
$$
\pr(\tau^{\textbf{1},\textbf{2}}_{N}>\beta_{N}s)\pr(\tau^{\textbf{1},\textbf{2}}_{N}>\beta_Nt)=\pr(\tau^{\textbf{1},\textbf{2}}_{N}>\beta_{N}s;\tau^{\textbf{1},\textbf{2};\beta_N s}_{N}>\beta_Nt),
$$
which gives
\begin{align}\label{Teo1Eq3}
\begin{split}
h_N&=\pr(\tau^{\textbf{1},\textbf{2}}_{N}>\beta_{N}s;\tau^{\textbf{1},\textbf{2};\beta_N s}_{N}>\beta_Nt;\zeta^{\textbf{1},\textbf{2},N}_{\beta_{N}s}\neq\zeta^{\textbf{1},\textbf{2},N}_{\beta_{N}s-c_{N},c_{N}}\text{ or }\zeta^{\textbf{1},\textbf{2},N}_{\beta_{N}s-c_{N},2c_{N}}\neq\zeta^{\textbf{1},\textbf{2},N}_{\beta_N s, c_{N}})\\&\leq \pr(\tau^{\textbf{1},\textbf{2}}_{N}>\beta_{N}s;\zeta^{\textbf{1},\textbf{2},N}_{\beta_{N}s}\neq\zeta^{\textbf{1},\textbf{2},N}_{\beta_{N}s-c_{N},c_{N}})+\pr(\zeta^{\textbf{1},\textbf{2},N}_{\beta_{N}s-c_{N},2c_{N}}\neq\zeta^{\textbf{1},\textbf{2},N}_{\beta_N s, c_{N}}).
\end{split}
\end{align}
Observe that by the Markov property we have that 
$$
\pr(\tau^{\textbf{1},\textbf{2}}_{N}>\beta_{N}s;\zeta^{\textbf{1},\textbf{2},N}_{\beta_{N}s}\neq \zeta^{\textbf{1},\textbf{2},N}_{\beta_{N}s-c_{N},c_{N}})\leq \underset{\zeta_0\in \mathcal{C}}{\sup}\,\pr^{\zeta_0}(\tau^{\zeta_0}_N>c_N;\zeta^{\zeta_0}_{c_N}\neq\zeta^{\textbf{1},\textbf{2},N}_{c_N})
$$ and
$$\pr(\tau^{\textbf{1},\textbf{2},\beta_N s-c_N}_N\geq 2 c_N;
\zeta^{\textbf{1},\textbf{2},N}_{\beta_{N}s-c_{N},2c_{N}}\neq\zeta^{\textbf{1},\textbf{2},N}_{\beta_N s, c_{N}})
\leq \underset{\zeta_0\in \mathcal{C}}{\sup}\,\pr^{\zeta_0}(\tau^{\zeta_0}_N>c_N;\zeta^{\zeta_0}_{c_N}\neq\zeta^{\textbf{1},\textbf{2},N}_{c_N}).
$$
Thus, in \eqref{Teo1Eq3} we have that 
$$
h_N\leq 2\nu^{N}+\pr(\tau^{\textbf{1},\textbf{2}}_{N}\leq2c_{N}).
$$
Therefore, $h_N$ converges to zero when $N$ goes to infinity. From this we deduce \eqref{Teo1Eq**}.

Now, to prove \eqref{Teo1Eq***} we observe that by the Markov property and \eqref{Eqnu} we have that
\begin{align*}
&\pr(\tau^{\textbf{1},\textbf{2}}_N>\beta_N (t+s))\\
&\leq\pr(\tau^{\textbf{1},\textbf{2}}_{N}>\beta_N t)\pr(\tau^{\textbf{1},\textbf{2}}_{N}>\beta_N s)+\underset{\zeta_0\in\mathcal{C}}{\sup\,}\pr(\zeta^{\textbf{1},\textbf{2},N}_{\beta_Nt}\neq\zeta^{\zeta_0,N}_{\beta_N t} ; \tau^{\zeta_0}_{N}>\beta_Nt)\pr(\tau^{\textbf{1},\textbf{2}}_{N}>\beta_N s)\\
&\leq\pr(\tau^{\textbf{1},\textbf{2}}_{N}>\beta_N t)\pr(\tau^{\textbf{1},\textbf{2}}_{N}>\beta_N s)+ \nu^N.
\end{align*}
Thus, we can take $h_N'=\nu^{N}$, and the proof is complete.
\end{proof}

\section{Proof of Theorem \ref{Theorem1}}\label{Section6}
In this section, we prove Theorem \ref{Theorem1},  which states  the asymptotic behavior of $\{\log\tau^{\textbf{1},\textbf{2}}_N\}_{N}$.  
Before the proof of the theorem, we present two technical results.   Proposition below is a modification of Proposition \ref{Prop2} which is suitable for our purpose.
\begin{proposition}\label{Prop3}There exists $0<c<1$ such that for every $K$
 \begin{equation}\label{Prop3 eq0} 
   \underset{\zeta_0 \in \mathcal{C}}{\sup\,}  \pr(\tau^{\zeta_0} _N>2N^2K a_{2N}; \nexists\, t\leq 2N^2K a_{2N}: \{x: \zeta^{\zeta_0,N}_{t}(x)=2\}\subset [1,N])\leq c^{KN}
 \end{equation}
 for $N$ large enough.
 \begin{proof}
 Let $s_k=k2N K a_{2N}$ for $1\leq k \leq N$. We observe that the same argument used for  the inclusion \eqref{Prop 2 eq1} leads to
 \begin{align}\label{Prop 3 eq4}
 \begin{split}
 &\left \lbrace \begin{array}{l}\tau^{\zeta_0}_N>2N^2K a_{2N};  \exists\hspace{0.1cm} y_{k} \in A_2, z_k \in A_4 : \zeta^{\zeta_0,N}_{s_{k}}(y_k)=2\\\zeta^{\zeta_0,N}_{s_{k}}(z_k)=1
 \text{ and } (y_k,s_k), (z_k,s_k) \text{ are an }N\text{-\emph{barrier}}\end{array}\right\rbrace
 \\&\hspace{5cm} \subset\{\{x: \zeta^{\zeta_0,N}_{s_k+S}(x)=2\}\subset [1,N]\}.
\end{split} \end{align}
To see this, first we fix a configuration in the event on the left member of \eqref{Prop 3 eq4}. Now, since $(z_k,s_k)$ is an $N$-\emph{barrier}, we have that if $\zeta^{\zeta_0,N}_{s_k+S}(x)\neq 0$ for a site $x\in[-N+1,0]$, then  $(x,s_k+S)$ is connected with $(z_k,s_k)$ inside $[-N+1,0]$ and by the priority of the particles of type $1$ in $[-N+1,0]\times[0,\infty)$, we have that $\zeta^{\zeta_0,N}_{s_k+S}(x)=1$. By  the same reasoning, we have that if  $x'\in[1,N]$ and $\zeta^{\zeta_0,N}_{s_k+S}(x')\neq 0$,  then $(x',s_k+S)$ is connected with $(y_k,s_k)$ inside $[1,N]$ and by the priority of the particles of type $2$ in $[1,N]\times[0,\infty)$, it holds that $\zeta^{\zeta_0,N}_{s_k+S}(x)=2$. Summing up, at time $s_k+S$ every site occupied in $[-N+1,0]$ is occupied by a particle of type $1$ and every site occupied in $[1,N]$ is occupied by a particle of type $2$, which yields \eqref{Prop 3 eq4}.

Now, we observe that \eqref{Prop 3 eq4} implies
\begin{align}\label{Prop 3 eq4*}
 \begin{split}
 &\underset{1\leq k \leq N}{\bigcup}\left \lbrace \begin{array}{l}\tau^{\zeta_0}_N>2N^2K a_{2N};  \exists\hspace{0.1cm} y_{k} \in A_2, z_k \in A_4 : \zeta^{\zeta_0,N}_{s_{k}}(y_k)=2\\\zeta^{\zeta_0,N}_{s_{k}}(z_k)=1
 \text{ and } (y_k,s_k), (z_k,s_k) \text{ are an }N\text{-\emph{barrier}}\end{array}\right\rbrace
 \\&\hspace{5cm} \subset\underset{0\leq t\leq 2N^2K a_{2N}}{\bigcup}\{\{x: \zeta^{\zeta_0,N}_{t}(x)=2\}\subset [1,N]\}.
\end{split} \end{align}
Therefore, to conclude \eqref{Prop3 eq0} it is enough to prove 
\begin{equation}\label{Prop 3 eq5}
\underset{\zeta_0 \in \mathcal{C}}{\sup\,} \pr\left(\tau^{\zeta_0} _N>2N^2K a_{2N}\cap \bigcap_{k=1}^{KN}\left \lbrace\begin{array}{l}   \forall y \in A_2 \hspace{0.1cm}\zeta^{\zeta_0,N}_ {s_k}(y)\neq2 \text{ or }(y,s_k) \\\text{ is not an } N\text{-\emph{barrier}}\end{array}\right \rbrace\right)\leq c^{KN}.
\end{equation}
We observe that the left member in \eqref{Prop 3 eq5} is  the same as the left member of \eqref{Prop 2 eq4*}, with the only difference that in this case we are intersecting $KN$ events instead of $N$. Thus, the same procedure used to get the bound  $c^N$ for the left member of \eqref{Prop 2 eq4*}  can be applied to obtain \eqref{Prop 3 eq5} (see Proposition \ref{Prop2}).
 \end{proof}
\end{proposition}
In the  next lemma, we use the following limit
\begin{equation}\label{Fact(ii)}
\underset{N\rightarrow \infty}{\lim}\frac{1}{N}\log(\pr(T^{[1,N]}<\infty))=-c_{\infty},
\end{equation}where $c_{\infty}$ is as in Remark \ref{Remark1} and $T^{[1,N]}$ is defined in \eqref{timext}.
This result is proved for $R=1$ in Lemma $3$ of \cite{Schonmman-Durrett}. Since  every step of this proof  can be applied for the case $R>1$, we assume \eqref{Fact(ii)} without proving it.
\begin{lemma}\label{LemmaMountford2} There exists $\theta>0$ such that
$$
\underset{N\rightarrow \infty}{\liminf}\frac{1}{N}\log(\pr(T^{[ 1,N ]}< \theta N))\geq -c_{\infty}.
$$
\end{lemma}
\begin{proof}

Observe that for any $\theta>0$
$$
\pr(T^{[1,N]}< \infty)=\pr(T^{[1,N]}< \theta N)+\pr(\theta N<T^{[1,N]}< \infty).
$$
Using \eqref{Fact(i)Durrett} for $t=\theta N$ we have
\begin{equation}\label{Teo 1 Eq 1}
\pr(T^{[1,N]}< \infty)\leq  e^{-\theta N \Hat{c}} + \pr(T^{[1,N]}<\theta N).
\end{equation}
By \eqref{Fact(ii)} and \eqref{Teo 1 Eq 1}, for all $\epsilon>0$  there exists an $n$ such that for all $N>n$
$$
e^{-(c_{\infty}+\epsilon)N}-e^{-\theta N \Hat{c}} \leq  \pr(T^{[1,N]}<\theta N),
$$
which implies
$$
-(c_{\infty}+\epsilon)N+\log(1+e^{-(\theta \Hat{c}-c_{\infty}-\epsilon)N})\leq \log \pr(T^{[1,N]}<\theta N).
$$
Taking $\theta>c_{\infty}/\Hat{c}$, for every $\epsilon>0$ we have 
$$
-(c_{\infty}+\epsilon)\leq \underset{N\rightarrow \infty}{\liminf}{\frac{1}{N}\log \pr(T^{[1,N]}<\theta N)}.
$$
\end{proof}
\begin{proof}[Proof of Theorem \ref{Theorem1}] First, for a fixed but arbitrary $\epsilon>0$  we will prove that 
\begin{equation}\label{Teo 1 Eq 2}
\underset{N\rightarrow \infty}{\lim}\pr(\tau^{\textbf{1},\textbf{2}}_N >k_{N}e^{(c_{\infty}+2\epsilon)N})=0,
\end{equation}
where $k_{N}=2NK^{*}a_{2N}+\theta N$,   $K^{*}=\left \lceil\frac{2c_{\infty}+  \epsilon}{\log(1/c)} \right\rceil$,  $a_{2N}$ is  as in Proposition \ref{Prop0}, $\theta$ is as in Lemma \ref{LemmaMountford2} and $c$ is as in Proposition \ref{Prop3}.

To do this, we observe that by the Markov property, for every $n \in \mathbb{N}$  it holds that
\begin{equation}\label{Teo 1 Eq 3}
\pr(\tau^{\textbf{1},\textbf{2}}_{N}>n k_{N})
\leq (\underset{\zeta_0 \in \mathcal{C}}{\sup\, }\pr(\tau^{\zeta_0}_{N}>k_{N}))^{n}.
\end{equation}
Now, we observe that by   \eqref{Prop3 eq0} for every  $\zeta_0 \in \mathcal{C}$ we have that
\begin{align}\label{Teo 1 Eq 4}\begin{split}
&\pr(\tau^{\zeta_0}_{N}>k_{N})\leq c^{K^{*}N}+\\&\hspace{2.5cm}+\pr(\tau^{\zeta_0}_{N}>k_{N};\exists \hspace{0.1cm} t\leq k_{N}-\theta N : \{x: \zeta^{\zeta_0,N}_{t}(x)=2\}\subset [1,N]).
\end{split}
\end{align}
To deal with the probability in the right term of \eqref{Teo 1 Eq 4}, we define the following stopping time
$$
S^*=\inf\{t:\{x:\zeta^{\zeta_0,N}_{t}(x)=2\}\subset [1,N]\}.
$$
Using the strong Markov property for this stopping time and the atractiveness of the classical contact process we have that for $N$ large enough 
\begin{align}\label{Teo 1 Eq 6}
\begin{split}
&\pr(\tau^{\zeta_0}_{N}>k_{N};\exists \hspace{0.1cm} t\leq k_{N}-\theta N : \{x: \zeta^{\zeta_0,N}_{t}(x)=2\}\subset [1,N])\\
&=\pr(\tau^{\zeta_0}_{N}>k_{N};S^*\leq k_{N}-\theta N)\\
&\leq \pr(T^{[1,N]}>\theta N)\leq 1-e^{-(c_{\infty}+\epsilon)N},
\end{split}
\end{align}
where the last inequality use Lemma \ref{LemmaMountford2}.
 Substituting  \eqref{Teo 1 Eq 6} into \eqref{Teo 1 Eq 4} we obtain that
\begin{align}\label{Teo 1 Eq 6*}
\begin{split}
\underset{\zeta_0 \in \mathcal{C}}{\sup \,}\pr(\tau^{\zeta_0}_{N}>k_{N})&\leq c^{K^{*}N}+  1-e^{-(c_{\infty}+ \epsilon)N}\\
&\leq e^{-(2c_{\infty}+2\epsilon)N}+1-e^{-(c_{\infty}+ \epsilon)N}.
\end{split}
\end{align}
Thus, by \eqref{Teo 1 Eq 3} and \eqref{Teo 1 Eq 6*}, for $N^{*}=\left \lfloor e^{(c_{\infty}+2\epsilon)N} \right\rfloor$ we have
\begin{align}\label{Teo 1 Eq 10*}
\begin{split}
\pr(\tau^{\textbf{1},\textbf{2}}_{N}>k_{N}e^{(c_{\infty}+2\epsilon)N})&\leq\pr(\tau^{\textbf{1},\textbf{2}}_{N}>k_{N}N^{*})\\
&\leq (1-e^{-(c_{\infty}+\epsilon)N}+e^{-(2c_{\infty}+2\epsilon)N} )^{N^{*}}. 
\end{split}
\end{align}
By our choice of $N^*$ we have that the right member of \eqref{Teo 1 Eq 10*} converges to zero when $N$ goes to infinity.
Thus, we have proved \eqref{Teo 1 Eq 2}. Now, observe that \eqref{Teo 1 Eq 2} implies 
\begin{equation}\label{Teo 1 Eq 9*}
\underset{N\rightarrow\infty}{\lim}\pr\left(\frac{1}{N}\log(\tau^{\textbf{1},\textbf{2}}_N)>c_{\infty}+3\epsilon\right)=0.
\end{equation}

Therefore, to conclude the proof of the theorem we only need to state that for every $\epsilon>0$
\begin{equation}\label{Teo 1 Eq 8*}
\underset{N\rightarrow\infty}{\lim}\pr\left(\frac{1}{N}\log(\tau^{\textbf{1},\textbf{2}}_N)<c_{\infty}-\epsilon\right)=0.
\end{equation}
For this purpose, observe  that  $\tau^{\textbf{1},\textbf{2}}_{N}$ is stochastically larger than the minimum of two independent variables with the same law of  $T^{\textbf{1}}_N$. Then, we have that 
\begin{align}\label{Teo 1 Eq 7*}
\begin{split}
\pr\left(\frac{1}{N}\log(\tau^{\textbf{1},\textbf{2}}_N)<c_{\infty}-\epsilon\right)
&\leq\pr\left(\frac{1}{N}\log(\min\{T^{\textbf{1}}_N;\tilde{T}^{\textbf{1}}_N\})<c_{\infty}-\epsilon\right)\\
&=\pr(\min\{T^{\textbf{1}}_N;\tilde{T}^{\textbf{1}}_N\}<e^{(c_{\infty}-\epsilon)N})\\
&=\pr(T^{\textbf{1}}_N<e^{(c_{\infty}-\epsilon)N})^{2},
\end{split}
\end{align}
where $T^{\textbf{1}}_N$ and $\tilde{T}^{\textbf{1}}_N$ are i.i.d. By \eqref{Prop 0 eq4}, the limit of the last term in \eqref{Teo 1 Eq 7*} is zero, which implies \eqref{Teo 1 Eq 8*}.

Clearly, from \eqref{Teo 1 Eq 9*} and \eqref{Teo 1 Eq 8*} the theorem follows. 
\end{proof}
In the next remark, we discuss what happens after the first type of particle dies out.
During this remark, we denote by $\xi^{A}_{[-N+1,N]}(t)$ the classical contact process with initial configuration $A$ and $T^{A}_{[-N+1,N]}$ the time of extinction of this process. For the special case $A=[-N+1,N]$, we write  $\xi^{\textbf{1}}_{[-N+1,N]}(t)$ and $T^{\textbf{1}}_{[-N+1,N]}$.
\begin{remark}
Let $\tilde{T}^{\textbf{1}}_{2N}$  be the time of extinction of both particles, that is
$$
\tilde{T}^{\textbf{1}}_{2N}=\inf\{t>0: \zeta^{\textbf{1},\textbf{2},N}_t=\emptyset\}.
$$
If we ignore   the existence of both types of particles, the dynamic of the process is the same as the classical contact process. Therefore, $\tilde{T}^{\textbf{1}}_{2N}$ has the same distribution as $T^{\textbf{1}}_{[-N+1,N]}$ and, consequently, Remark \ref{Remark1} implies that for $\epsilon>0$ we have
 \begin{equation}\label{Eq1Rem3}
 \underset{N\rightarrow \infty}{\lim} \pr\left(e^{(c_{\infty}-\epsilon)2N}< \tilde{T}^{\textbf{1}}_{2N} <e^{(c_{\infty}+\epsilon)2N}\right)=1.
\end{equation}
 Moreover, observe that after $\tau^{\textbf{1}, \textbf{2}}_N$ the process behaves like the classical contact process, since after that time there is only one type of  particle. Observe also that combining \eqref{Eq1Rem3} with Theorem \ref{Theorem1} we obtain 
 \begin{equation}\label{Eq2Rem3}
  \underset{N\rightarrow \infty}{\lim} \pr\left(e^{(c_{\infty}-\epsilon)2N}-e^{(c_{\infty}+\epsilon)N}< \tilde{T}^{\textbf{1}}_{2N}-\tau^{\textbf{1}, \textbf{2}}_N\right)=1.
 \end{equation}

Furthermore, after the extinction of one of the types of particle, the surviving type behaves like the classical contact process.  Thus, $\tilde{T}^{\textbf{1}}_{2N}-\tau^{\textbf{1}, \textbf{2}}_N$ has the same law of $T^{A_N}_{[-N+1,N]}$, where $A_{N}=\zeta^{\textbf{1},\textbf{2},N}_{\tau^{\textbf{1},\textbf{2}}_N}$. Using \eqref{Prop0itemi}   we have that
 \begin{equation}\label{Eq2'Rem3}
 \underset{N\rightarrow \infty}{\lim}\pr(\xi^{A_N}_{a_{2N}}=\xi^{\textbf{1}}_{[-N+1,N]}(a_{2N});T^{A_N}_{[-N+1,N]}>a_{2N})=0
 \end{equation}
  and by \eqref{Eq2Rem3} and the limit \eqref{Eq2'Rem3}  we have that
  $$
 \underset{N\rightarrow \infty}{\lim}\pr(\xi^{A_N}_{a_{2N}}=\xi^{\textbf{1}}_{[-N+1,N]}(a_{2N}))=0.
 $$
  Therefore 
 $$
  \underset{N\rightarrow \infty}{\lim}\pr(T^{A_N}_{[-N+1,N]}=T^{\textbf{1}}_{[-N+1,N]})=1
 $$
 and by Remark \ref{Remark1} we obtain that 
$$\underset{N\rightarrow \infty}{\lim}\frac{1}{2N}\log(\tilde{T}^{\textbf{1}}_{2N}-\tau^{\textbf{1},\textbf{2}}_N)=\underset{N\rightarrow \infty}{\lim}\frac{1}{2N}\log T^{\textbf{1}}_{[-N+1,N]}=c_{\infty}\text{ in Probability}.$$ 
\end{remark}

\textbf{Acknowledgements:} This work is part of the author’s Ph.D. thesis, written  at UFRJ and supported by CAPES. The author thanks  Majela Pent\'on Machado for a careful reading of this work
and the many constructive suggestions which improved the exposition considerably. The author also thanks Edgar Matias da Silva for the helpful comments during the preparation of this paper. The author thanks Juan Carlos Salcedo Sora for his support and patience. The author also thanks to the two anonymous referees for their several helpful comments on this paper.
\bibliographystyle{acm}
\addcontentsline{toc}{part}{Bibliography}
\bibliography{reference}
Mariela Pent\'on Machado, Instituto de Matem\'atica. Universidade Federal do Rio de Janeiro, RJ, Brazil. Email: mariela@dme.ufrj.br.
\end{document}